%% file: RWREQLT4.tex
\documentclass[12pt,leqno]{article}
\usepackage{graphicx,psfrag,amssymb,amsmath,amsthm,makeidx}
\usepackage{epstopdf}

\input{abbrev.tex}

\newtheorem{theorem}{Theorem}
\newtheorem{lemma}{Lemma}[section]
\newtheorem{corollary}{Corollary}

\theoremstyle{definition}
\newtheorem*{remark}{Remark}
\newtheorem*{definition}{Definition}

\numberwithin{equation}{section}

\begin{document}

\title{Quenched limit theorems for nearest neighbour random walks in 1D random environment}

\author{D. Dolgopyat$^{1}$ and I. Goldsheid$^{2}$}
\footnotetext[1]{Department of
Mathematics and Institute of Physical Science and Technology,\\
University  of Maryland, College Park, MD, 20742, USA}
\footnotetext[2]{School of Mathematical Sciences, Queen
Mary University of London, London \\E1 4NS, Great Britain}
\date{}

\maketitle

\begin{abstract}
It is well known that random walks in one dimensional random environment
can exhibit subdiffusive behavior due to presence of traps. In this paper
we show that the passage times of different traps are asymptotically independent exponential
random variables with parameters forming, asymptotically,
a Poisson process. This allows us to prove weak quenched limit theorems
in the subdiffusive regime where the contribution of traps plays the dominating role.
\end{abstract}

\section{Introduction}\label{Intro}
Let $\omega=\{p_i\}, i\in\integers$ be an i.i.d. sequence of random
variables, $0<p_i<1$. The sequence $\omega$ is called {\it
environment} (or {\it random environment}). Let $(\Omega,\bP)$ be the
corresponding probability space with $\Omega$ being the set of all
environments and $\bP$ the probability measure
on $\Omega$. The expectation with respect to this measure will be denoted by $\bE$.
Given an $\omega$ we define a random walk $X=\{X_n,\ n\ge0\}$ on
$\integers$ in the environment $\omega$ by setting $X_0=0$ and
$$ \Prob_\omega(X_{n+1}=X_n+1|X_0\dots X_n)=p_{X_n}\quad
\Prob_\omega(X_{n+1}=X_n-1|X_0\dots X_n)=q_{X_n}$$ where
$q_n=1-p_n.$ Denote by $\mathfrak{X}=\{X\}$ the space of all trajectories
of the walk starting from zero. A \textit{quenched} (fixed) environment $\omega$
thus provides us with a conditional probability measure $\Prob_\omega$ on
$\mathfrak{X}$. The expectation with respect to $\Prob_\omega$
will be denoted by $\EXP_\omega$. In turn, these two
measures naturally generate the so called {\it annealed} measure on
the direct product $\Omega\times\mathfrak{X}$ which is a semi-direct product $\mathrm{P}:=\bP\ltimes\Prob_\omega$.
However, with a very slight abuse of notation,
$\bP$ and $\bE$ will also denote the latter measure and the
corresponding expectation; the exact meaning of the corresponding
probabilities and expectations will always be clear from the
context. The term {\it annealed walk} will be used to discuss
properties of the above random walk with respect to the annealed
probability.

From now on we assume that

(A) $\bE(\ln (p/q))>0.$

(B) $\bE\left(\frac{q}{p}\right)^s=1$ for some $s>0.$

(C) There is a constant $\eps_0$ such that
$\eps_0\leq p_n\leq 1-\eps_0$ with probability~1.

(D) The support of $\ln(q/p)$ is non-arithmetic.

Assumption (A) implies (see \cite{So}) that $X_n\to\infty$ with
probability 1. Assumption (B) means that even though the walker goes
to $+\infty$ there are some sites where the drift points in the
opposite direction. We note that (A) and (B) are essentially equivalent to each other.
Indeed,  since
$\bE\left(\frac{q}{p}\right)^h$ is a convex function of
$h$, (B) implies (A). On the other hand, the existence of finite $s$ in (B) follows from (A)
 if and only if $\bP(q>p)>0$. It is convenient to have both these conditions on the list for
reference purposes.

(C) is a standard ellipticity assumption which prevents
the walker from getting stuck at finitely many vertices for a long time.

(D) is a technical assumption which we don't use in
our proofs but which is used in the proof of Lemma \ref{LmRen}
borrowed from \cite{K1}. It is satisfied by a generic distribution
of $p_n.$

We will be mostly interested in the case $s\in (0,2]$ which implies that the annealed
distribution of $X_n$ does not satisfy the standard Central Limit
Theorem (\cite{KKS}).
Since $X_n$ is transient it looks monotonically increasing on a large
scale and hence it makes sense to study the hitting time $\tilde{T}_N:=\min(n: X_n=N)$
which can roughly be viewed as the inverse function of $X_n$.
This approach was used already in the pioneering papers \cite{So} and \cite{KKS}.
In particular, in \cite{KKS} the annealed behavior of $X_n$ was derived from that of  $\tilde{T}_N$.
The latter is described by the following
\begin{theorem} (\cite{KKS})
\label{ThAnn} The annealed random walk $X$ has the following properties:

(a) If $s<1$ then the distribution of $\frac{\tilde{T}_N}{N^{1/s}}$ converges to a stable law with index $s.$

(b) If $1<s<2$ then there is a constant $u$ such that the distribution of
$\frac{\tilde{T}_N-N u}{N^{1/s}}$ converges to a stable law with index $s.$

(c) If $s>2$ then there is a constant $u$ such that the distribution of
$\frac{\tilde{T}_N-N u}{N^{1/2}}$ converges to a normal distribution.

(d) If $s=1$ then there is a sequence $u_N\sim c N \ln N$ such that the distribution of
$\frac{\tilde{T}_N-u_N}{N}$ converges to a stable law with index $1.$

(e) If $s=2$ then there is a constant $u$ such that the distribution of
$\frac{\tilde{T}_N-N u}{\sqrt{N\ln N}}$ converges to a normal distribution.
\end{theorem}

The proof of this theorem given in \cite{KKS} makes use of the connection
between random walks in random environment and branching processes.
Another proof of Theorem \ref{ThAnn} was given in \cite{ESZ, ESTZ}.
These papers make use of the notion of potential
introduced by Ya. G. Sinai in \cite{S} for the study of the recurrent case
(when $\bE(\ln(p/q))=0$).

The results for quenched limits (that is when a typical environment is fixed) are relatively recent.
To prove an almost sure quenched limit theorem for $\tilde{T}_N$ one can make use of the
representation
\begin{equation}
\label{HitTimes}
\tilde{T}_N=\sum_{i=1}^{N}\tau_i,
\end{equation}
where $\tau_i$ is the time the walk starting
from $i-1$ needs in order to reaches $i$ for the first time. The advantage of this approach
is due to the fact that if the environment $\omega$ is fixed then $\tau_i$ are independent random variables
and this was used by many authors starting from the pioneering paper \cite{So}.

If $s>2$ then one can prove the almost sure Central Limit Theorem (CLT) for $\tilde{T}_N$
checking that the sequence $\{\tau_i\}$ in \eqref{HitTimes} satisfies the Lindeberg condition for almost all $\omega$
(and for that one only needs the environment $\{p_i\}$ to be stationary, see e.g. \cite{G1}).
Proving the CLT for $X_n$ in this regime is a more delicate matter and this was done
in \cite{G1} for several classes of environments (including the i.i.d. case) and independently in
\cite{P1} for the i.i.d. environments. It has to be mentioned that, in the case of i.i.d.
environments, it is easy to derive the annealed CLT from the related quenched CLT
but this may not be easy for other classes of environments and in fact may not
always be true.

For $s<2,$ an important step forward was made in \cite{P2} and \cite{PZ}
where it was proved that for almost all $\omega$
no non-trivial distributional limit of $\frac{\tilde{T}_N-u_N}{v_N}$ exists for
any choice of sequences $u_N$ and $v_N$ (which may depend on $\omega$). The result is thus
negative: it is impossible to have almost sure quenched limit theorems in this regime.

One of the main goals of this paper is to complete the picture. We show that
$\tilde{T}_N$ viewed as a function of two(!) random parameters, $X$ and $\omega$ (the trajectory of
the walk and the environment), does nevertheless exhibit a limiting
behaviour as $N\to\infty$ which for $0<s<2$ can be described
explicitly in terms of a point Poisson process (Theorem \ref{ThMain}).
Namely, it turns out that for large fixed $N$ and
$\omega\in\Omega_N$ (where $\bP(\Omega_N)\to 1$ as $N\to\infty$) the properly normalizes
$\tilde{T}_N$ is a linear combination of independent exponential
random variables with coefficients of this combination depending only on $\omega$
and forming a point Poisson process. As a corollary, one obtains the results from
\cite{P2} and \cite{PZ} as well as anew proof of Theorem \ref{ThAnn}.
In the case $s=2$ we show that the CLT holds (Theorem \ref{ThQLT}); however, we also show that
it does not hold for almost all $\omega$ but rather just for $\omega\in\Omega_N$
\footnote{Unlike Theorem \ref{ThAnn}, Theorems \ref{ThMain} and \ref{ThQLT} do
not include the case $s>2$. For
reasons discussed above this is not necessary.
However, we explain at the end of Section \ref{ScS2} that it is
not difficult to adapt the argument of that section to handle also
the diffusive regime.}.

The backbone of our approach is formed by the study of occupation times;
such studies were initiated in \cite{R, Shi, GS}. In view of this technique it is more natural to
consider the occupation time $T_N$ of the interval $[0,\,N)$ rather than $\tilde{T}_N$.
These two random variables have the same asymptotic behaviour  (see Lemma \ref{LmS-T})
and therefore the results for $\tilde{T}_N$ follow easily from those
for occupation times.

The main difference between our and other existing approaches is that:

$-$ We introduce a Poisson process describing the ``trapping properties" of the environment.

$-$ This process allows us to explicitly separate the contribution
to the occupation time (or, equivalently, hitting time) coming from the
environment and the walk (and thus prove Theorem \ref{ThMain}).

$-$ It also allows us to answer some other interesting questions
about the limiting behaviour of the walk (e.g., about the limiting
behaviour of the distribution of the maximal occupation times, Theorem \ref{ThRec}).

Similar results are valid in a more general setting of random walks in random environment on a
strip and in particular for walks with bounded jumps.
This will be a subject of a separate paper.

The layout of the paper is the following. In Section \ref{ScResults} we
state our main results. In Section \ref{ScPrelim} we collect background
information and prove some auxiliary results. In Section \ref{ScProof}
we deduce Theorem \ref{ThMain} dealing with the case $s<2$ from the fact that the set of sites
with high expected number of visits has asymptotically Poisson distribution
(Lemma \ref{LmPVis}). The proof of Lemma \ref{LmPVis} itself is given in Section
\ref{ScPVis}. The case when $s=2$ (Theorem \ref{ThQLT}) requires a different approach (namely, we use
big block-small block method of Bernstein) which is presented in Section \ref{ScS2}.
In Section \ref{ScMOT} we explain how to modify the proof of Theorem \ref{ThMain} to obtain
Theorem \ref{ThRec}. Finally in the appendix we derive some previously
known theorems from our results.

After completing the paper we learned that Corollary \ref{CrQLT} was proved independently by
J. Peterson and G. Samorodnitsky using a different approach. After a discussion with them we
we agreed that  our paper and \cite{PS} should be published separately in order to make both approaches available to the public.

\section{Main results}\label{ScResults}
Throughout the paper the following definitions and notations will be used.
\begin{definition} The {\it occupation time} $T_N $ of the interval $[0,N)$
is the total time the walk $X_n$ starting from $0$ spends on this (semi-open) interval
during its life time. In other words, $T_N=\#\{n:\ 0\le n< \infty,\ 0\le X_n\le N-1\}$
\end{definition}
\begin{remark} We thus use the following convention: starting from a site $j$
counts as one visit of the walk to $j$.
\end{remark}
The occupation time of a site $j$ is defined similarly and is denoted by $\xi_j$.
Observe that $T_N$ (and $\xi_j$) is equal
to the number of visits by the walk to $[0,N)$ (respectively, to site $j$).
Since our random walk is transient to the right, both $T_N$ and $\xi_j$ are,
$\bP$-almost surely, finite random variables.
It is clear from these definitions that
$$T_N=\sum_{j=0}^{N-1} \xi_j.$$
The following lemma shows that $T_N$ and the hitting time $\tilde{T}_N$
have the same asymptotic behaviour.
\begin{lemma} For any $\eps>0$
\label{LmS-T}
$$\bP\left(\frac{|T_N-\tilde{T}_N|}{N^{1/s}}>\eps\right)\to 0
\text{ as } N\to \infty. $$
\end{lemma}
\begin{proof}
It is easy to see that
$$\tilde{T}_N= \#\{n:\ 0 < n\le \tilde{T}_N,\ X_n\in[0,N-1]\} +\#\{n:\ 0 < n\le \tilde{T}_N,\ X_n<0\}$$
and
$${T}_N= \#\{n:\ 0\le n\le \tilde{T}_N,\ X_n\in[0,N-1]\} +\#\{n:\ n> \tilde{T}_N,\ X_n\in[0,N-1]\}.$$
Since the first terms in these formulae are equal, $|T_N-\tilde{T}_N|$ can be estimated above by a sum of two random variables: the number of visits to the left of 0 and the number of visits to the left of $N$ after $\tilde{T}_N$:
$$|T_N-\tilde{T}_N|\le\#\{n:\ n\ge0,\ X_n<0\}+\#\{n:\ n> \tilde{T}_N,\ X_n<N\} $$
The first term in this estimate is bounded for $\bP$-almost
all $\omega$. Since $\tilde{T}_N$ is a hitting time, the second
term has, for a given $\omega$, the same distribution as
$\#\{n:\ n> 0,\ X_n<N\,|\,X_0=N\} $ (due to the strong Markov property).
Finally, the latter is a stationary sequence with respect to the
annealed measure and therefore is stochastically bounded. Hence the Lemma.
\end{proof}
\begin{remark} The difference between $T_N$ and $\tilde{T}_N$ is thus negligible and yet
there is a sharp contrast between their presentations by sums introduced above.
Namely, unlike the $\tau_i$'s, the $\xi_j$'s are not independent. Moreover, as we shall see
below, there are whole random regions on $[0,N]$ where the knowledge of just one $\xi_j$
essentially determines the values all the others. In fact, namely this strong interdependence
of $\xi_j$'s implies some of the main results of this paper.
\end{remark}
From now on we shall deal mainly with $ \mathfrak{t}_N$ which is the normalized version of $T_N$:
$$ \mathfrak{t}_N=\begin{cases} \frac{T_N}{N^{1/s}} & \text{if } 0<s<1, \cr
\frac{T_N-\EXP_\omega(T_N)}{N^{1/s}} & \text{if } 1\le s<2, \cr
\frac{T_N-\EXP_\omega(T_N)}{\sqrt{N\ln N}} & \text{if } s=2. \cr
\end{cases} $$
It is also important and natural to have control over the $\mathbb{E}_{\omega}(T_N)$. The corresponding normalized quantity is defined as follows:
$$ \mathfrak{u}_N=\begin{cases} \frac{\EXP_\omega(T_N)}{N^{1/s}} & \text{if } 0<s<1, \cr
\frac{\EXP_\omega(T_N)-u_N}{N} & \text{if } s=1, \cr
\frac{\EXP_\omega(T_N)-\bE(T_N)}{N^{1/s}} & \text{if } 1< s<2, \cr
\frac{\EXP_\omega(T_N)-\bE(T_N)}{\sqrt{N\ln N}} & \text{if } s=2, \cr
\end{cases} $$
where $u_N$ is the same as in Theorem \ref{ThAnn}. Set
$$ F_N^\omega(x)=\Prob_\omega\left(\mathfrak{t}_N\leq x\right). $$
Then $x\to F_N^\omega(x)$ is a sequence of random
processes. The result from \cite{P2, PZ} cited above states that these processes are not concentrated near one point (at least for $0<s<2$). Nonetheless, the limiting behaviour of the sequence $\mathfrak{t}_N$
can be described in terms of a marked point Poisson process which we shall now introduce.

We start with a point Poisson process. Given a $\bc>0,$ let $\Theta=\{\Theta_j\}$ be a point  Poisson process
\footnote{For reader's convenience we
collect some facts about the Poisson processes in section \ref{SSPP}.}
on $(0, \infty)$ with intensity
$\frac{\bc}{\theta^{1+s}}.$
For a given collection of points $\{\Theta_j\}$  let $\{\Gamma_{\Theta_j}\}$
be a collection of i.i.d. random variables
with mean 1 exponential distribution which are thus labeled by the points $\{\Theta_j\}$.
In the sequel we shall use a concise notation $\{\Gamma_j\}$ for $\{\Gamma_{\Theta_j}\}$.
We can now consider a new process $(\Theta,\Gamma)=\left(\{\Theta_j,\Gamma_j\}\right)$ which is often
called the marked point Poisson process.
We note that $(\Theta,\Gamma)$ is in fact a point Poisson process on
$(0,\infty)\times (0,\infty)$ with intensity $\frac{\bc}{\theta^{1+s}}\times e^{-x}$.
We shall denote by $E_{\Theta}$, $\Var_{\Theta}$, etc. the relevant expectations, variances, etc.
Set
$$ Y=\begin{cases} \sum_j \Theta_j \Gamma_j & \text{if } 0<s<1 \cr
\sum_j \Theta_j \left(\Gamma_j-1\right) & \text{if } 1\le s<2 \cr
\end{cases}. $$
Observe that $Y$ is finite almost surely. Indeed, there are only finitely many points with $\Theta_j\geq 1.$
Next, if $0<s<1$ let
$$\tY=\sum_{\Theta_j<1} \Theta_j \Gamma_j . $$
Then
$$E_{\Theta}(\tY)=\int_0^1 \frac{\bc \theta d\theta}{\theta^{1+s}}=\frac{\bc}{1-s}<\infty. $$
In case $1\leq s<2$ let
$$\tY_\delta=\sum_{\delta<\Theta_j<1} \Theta_j (\Gamma_j-1). $$
Then $E_{\Theta}(\tY_\delta)=0$ and
$$ \Var_{\Theta}(\tY_\delta)=\int_\delta^1 \frac{\bc \theta^2 d\theta }{\theta^{1+s}}=\frac{\bc}{2-s}\left(1-\delta^{2-s}\right). $$
Denote by $\Theta^{(\delta)}$ a point Poisson Process
on $\mathbb{R}_\delta:=[\delta, \infty)$ with intensity $\frac{\bc}{\theta^{1+s}}$
and let $(\Theta^{(\delta)},\Gamma)$ be a marked process with $\Gamma$ being as
above. Obviously $(\Theta,\Gamma)$ corresponds to $\delta=0$ and
$(\Theta^{(\delta)},\Gamma)$ can be viewed as a projection of the former to a smaller
phase space.

Let $\mathfrak{F}_\delta$ be the set of all finite subsets of
$\mathbb{R}_\delta$ and $\Theta^{(N,\delta)}$ be a sequence of point processes
defined on the space of environments $\Omega$ and taking values in $\mathfrak{F}_\delta$. On a more formal level we can write $\Theta^{(N,\delta)}: \Omega\mapsto \mathfrak{F}_\delta$. The standard definitions of the relevant sigma-algebra and measurability can be found e.g. in \cite{SamT}. In the
constructions below such sequences will be arising in a natural way and it will always be
clear that the relevant mappings  are measurable.
Set $|\Theta^{(N, \delta)}|\equiv \Card(\Theta^{(N, \delta)})$. We need the following
\begin{definition}
A sequence of random point processes
$\Theta^{(N,\delta)}=\{\Theta_j^{(N,\delta)}\}$ defined
on $\Omega$ converges weakly to a Poisson process
$\Theta^{(\delta)}$ if for any $k\ge 1$ and any bounded continuous symmetric
function $H_k:\mathbb{R}_\delta^k\mapsto \mathbb{R} $ of $k$ variables
$$\lim_{N\to\infty}\mathbf{E}\left(H_k(\Theta^{(N, \delta)})\,I_{|\Theta^{(N, \delta)}|=k}\right)
=E_{\Theta}\left(H_k(\Theta)\,I_{|\Theta|=k}\right).$$
\end{definition}
Suppose next that $\Gamma^{(N,\delta)}$ is a collection of random variables defined
on $\Omega\times\mathfrak{X}$ and which can be labeled by the points of
$\Theta^{(N,\delta)}=\{\Theta^{(N,\delta)}_j\}$. In other words, for a given
$(\omega, X)\in \Omega\times\mathfrak{X}$ we suppose that
$\Theta^{(N,\delta)}=\{\Theta^{(N,\delta)}_j(\omega)\}$
and $\Gamma^{(N,\delta)}=\{\Gamma^{(N,\delta)}_{\Theta_j}(\omega,X)\}$.
Once again, as in the definition of
$\Theta^{(N,\delta)}$  we shall write
$\{\Gamma^{(N,\delta)}_{j}\}$ for $\{\Gamma^{(N,\delta)}_{\Theta_j}\}$.
Finally, the marked process $(\Theta^{(N,\delta)},\Gamma^{(N,\delta)})=
(\{\Theta^{(N,\delta)}_j,\Gamma^{(N,\delta)}_{j}\})$ can be viewed as a mapping
$$(\Theta^{(N,\delta)},\Gamma^{(N,\delta)}):
\Omega\times\mathfrak{X}\mapsto \mathfrak{F}_\delta\times \tilde{\mathfrak{F}},$$
where $\tilde{\mathfrak{F}}$ is the set of all finite subsets of $[0,\infty)$.
The weak convergence of this sequence of processes to $(\Theta^{(\delta)},\Gamma)$
is defined as above with the only
difference that now we have to deal with symmetric continuous functions
$H_k:\left(\mathbb{R}_\delta\times[0,\infty)\right)^{k}\mapsto \mathbb{R}$.
\begin{definition}
The component $\Gamma^{(N,\delta)}(\omega)$ is
said to be asymptotically independent of the environment if for any $k\ge 1$ and any bounded continuous
symmetric function $H_k:\mathbb{R}_0^k\mapsto \mathbb{R} $ of $k$ variables
$$\lim_{N\to\infty}\mathbf{E}\left[I_{|\Theta^{(N, \delta)}|=k}\left|\mathbb{E}_\omega\left(H_k(\Gamma^{(N, \delta)})\right)
-E_{\Gamma}\left(H_k(\Gamma)\,|\,|\Theta^{(\delta)}|=k\right)\right|\right]=0.$$
\end{definition}
Note that here $H_k(\Gamma^{(N, \delta)})$ is well defined because $|\Gamma^{(N, \delta)}|
=|\Theta^{(N, \delta)}|=k$. We can now state our main result.
\begin{theorem}
\label{ThMain} For $0<s<2$ and a $\delta>0$ there is a sequence
$\Omega_{N,\delta}\subset\Omega$ such that
$\lim_{N\to\infty}\bP(\Omega_{N,\delta})=1$ and a sequence of random point processes
$$(\Theta^{(N,\delta)},\Gamma^{(N,\delta)}):
\Omega\times\mathfrak{X}\mapsto \mathfrak{F}_\delta\times \tilde{\mathfrak{F}},$$
such that \newline
(i) The component $\Theta^{(N,\delta)}$ depends only on $\omega$ and converges weakly to a point Poisson process $\Theta^{(\delta)}$ on $[\delta,\infty)$ with intensity $\frac{\brc}{\theta^{1+s}}$ (with some constant $\brc>0$).
\newline
(ii) The component $\Gamma^{(N,\delta)}$ 
is asymptotically independent of the environment and converges weakly to a sequence of $|\Theta^{(N, \delta)}|$
i.i.d. exponential random variables with mean 1.
\newline
(iii) The $\mathfrak{t}_N$ and $\mathfrak{u}_N$ can be presented in the following form:
\newline
(a) If\, $0<s<1$ then for $\omega\in\Omega_{N,\delta}$
\begin{equation}\label{Th2a}
\mathfrak{t}_N=\sum_{j}\Theta_j^{(N,\delta)}\Gamma_j^{(N,\delta)} + R_N,\
\ \text{where}\ \  R_N\ge 0 \ \ \text{and}\ \ \bE(R_N)=O(\delta^{1-s})
\end{equation}
$$\mathfrak{u}_N=\sum_{j}\Theta_j^{(N,\delta)} + O(\delta^{1-s})$$
\newline
(b) If $s=1$ then there is $\Omega_{N,\delta}$ such
that for $\omega\in\Omega_{N,\delta}$ and a given $0<\kappa<1$
$$\mathfrak{t}_N=\sum_{j}\Theta_j^{(N,\delta)}(\Gamma_j^{(N,\delta)}-1) + R_N,\ \ \text{where}\ \  \bE\left[1_{\Omega_{N,\delta}}\mathbb{E}_{\omega}(R_N^2)\right]^\kappa=O(\delta^{2\kappa})$$
$$\mathfrak{u}_N=\sum_{j}\Theta_j^{(N,\delta)} - \brc\ln\delta+ R_N,\ \ \text{where}\ \  \bE(|R_N|^{2\kappa})=O(\delta^{2\kappa})$$
\newline
(c) If $1< s<2$ then for $\omega\in\Omega_{N,\delta}$
$$\mathfrak{t}_N=\sum_{j}\Theta_j^{(N,\delta)}(\Gamma_j^{(N,\delta)}-1) + R_N,\ \ \text{where}\ \  \bE\left[1_{\Omega_{N,\delta}}\mathbb{E}_{\omega}(R_N^2)\right]=O(\delta^{2-s})$$
$$\mathfrak{u}_N=\sum_{j}\Theta_j^{(N,\delta)}-\frac{\brc}{(s-1)\delta^{s-1}} + R_N,\ \ \text{where}\ \  \bE(R_N^2)=O(\delta^{2-s})$$
\end{theorem}
\begin{remark} Note that the dependence of $\Theta^{(N,\delta)}$ on $\omega$ persists as $N\to\infty$
whereas  $\Gamma^{(N,\delta)}$ becomes ``almost" independent of $\omega$. More precisely, for
$K\gg 1$ and sufficiently large $N$ the events $B_k:=\{|\Theta^{(N,\delta)}|=k\}$, $1\le k\le K$,
form, up to a set of a small probability, a partition of $\Omega$. Obviously
$$\lim_{N\to\infty} \bP\{|\Theta^{(N,\delta)}|=k\}=\frac{e^{-\tilde{c}\delta^{-s}}(\tilde{c}\delta^{-s})^k}{k!},$$
 where $\tilde{c}=\brc/s$.
In contrast,
if $\omega\in B_k$ then $\Gamma^{(N,\delta)}(\omega, X)$ is a collection of $k$ random variables which
converge weakly as $N\to\infty$ to a collection of $k$ i.i.d. standard exponential random variables.
Thus the only dependence of $\Gamma^{(N,\delta)}(\omega, X)$ on $\omega$ and $\delta$ which persists as $N\to\infty$ is reflected by the fact that
$|\Theta^{(N,\delta)}|=|\Gamma^{(N,\delta)}|$.
\end{remark}
Given a $\Theta$ let $F^\Theta$ be the conditional distribution
function of $Y$. The following statements are easy consequences of
Theorem \ref{ThMain}.

\begin{corollary}
\label{CrQLT}
(a) If $0<s<2,$ $s\neq 1$ then
$F^\omega_N$ converges weakly to $F^{\Theta}.$

(b) If $1< s<2$ then
$ \left(F^\omega_N, \frac{\EXP_\omega(T_N)-\bE(T_N)}{N^{1/s}}\right)$
converges weakly to
$$ \left(F^{\Theta}, \lim_{\delta\to0}\left(\sum_{\Theta_j>\delta} \Theta_j-\frac{\brc}{(s-1)\delta^{s-1}}\right)\right). $$

(c) If $s=1$ then there exists $u_N\sim \brc N\ln N$ such that
$ \left(F^\omega_N, \frac{\EXP_\omega(T_N)-u_N}{N}\right)$
converges weakly to
$$ \left(F^{\Theta}, \lim_{\delta\to0}\left(\sum_{\Theta_j>\delta}
\Theta_j-E(\delta)\right)\right)$$
where
$E(\delta)=\bE\left(\sum_{\delta<\Theta_j<1} \Theta_j\right)=-\brc\ln\delta.$
\end{corollary}

\begin{remark}
Similar limiting distribution were obtained in \cite{ST} for a simpler model of `random climbing'
where the particle moves forward with unit speed and with intensity 1 it slides back to a nearest point of intensity
$\lambda$ Poisson process.
\end{remark}

We also recovers the result of \cite{PZ}.

\begin{corollary}
\label{CrNoLimit}
For $0<s<2$ and $\bP$-almost every environment $\omega$ the sequence $\mathfrak{t}_N(\omega,X)$
has no limiting distribution as $N\to\infty$. Moreover, for $0<s<1$ and $\bP$-almost
every environment $\omega$ any distribution that can be obtained as a limit of finite linear
combinations $\sum_{j}a_j\Gamma_j$, where $a_j>0$,  can also be obtained as a weak limit of
$\mathfrak{t}_{N_k}(\omega, X)$ as $k\to\infty$, where $N_k$ depends on $\omega$.
\end{corollary}
The proof of this statement will be given in the Appendix.

\smallskip\noindent
We complete the picture by stating the result for the case $s=2$.
\begin{theorem}
\label{ThQLT} If $s=2$ then there are constants $D_1, D_2$ such that
$(\mathfrak{t}_N, \mathfrak{u}_N)$
converge weakly to $(\cN_1, \cN_2)$ where $\cN_1$ and $\cN_2$ are independent Gaussian random variables
with zero means and variances $D_1$ and $D_2$ respectively.
Apart of that, $\mathfrak{t}_N$ is asymptotically independent of the environment.
\end{theorem}
It is well known that the reason why the hitting times do not always
satisfy the Central Limit Theorem is the presence of traps which
slow down the particle. It will be seen in the proofs that
Theorems \ref{ThMain} and \ref{ThQLT} state that if traps are
ordered according to the expected time the walker spends inside the
trap then the asymptotic distribution of traps is Poissonian with
intensity  $\frac{\bc}{\theta^{1+s}}.$ This result holds regardless
of the value of $s$. However, if $s\geq 2$ then the time spent inside
the traps is smaller than the time spent outside of the traps.

\begin{remark}
For $s=2$ the fact that $\mathfrak{u}_N$
is asymptotically normal was proved in \cite{GLP} and so to
prove Theorem \ref{ThQLT} it is enough to show that for any $\eps>0$
\begin{equation}
\label{2Q1st}
\mathbf{P}\left(\sup_x
\left|F_N^\omega(x)-F_{\cN_1}(x)\right|>\eps\right)\to 0\ \hbox{ as }\  N\to\infty.
\end{equation}
Indeed $F_N^\omega$ and $\mathfrak{u}_N=\frac{\EXP_\omega(T_N)-\bE(T_N)}{\sqrt{N\ln N}}$
are evidently asymptotically independent since the distribution of the latter
depends only on the environment and the distribution of the former is asymptotically
the same for the set of $\omega$s of asymptotically full measure.
\end{remark}

Let as before $\xi_n$ be the number
of visits to $n$ and $\xi_N^*=\max_{[0,N]} \xi_n.$

\begin{theorem}
\label{ThRec}
If $s>0$ then
$\frac{\xi_N^*}{N^{1/s}}$ converges to $\max_j \brTheta_j$, where
$\brTheta$ is a  Poisson process on $(0,\infty)$ with intensity $\frac{\brbc}{\theta^{1+s}}$ for some constant $\brbc.$
Accordingly
$$ \bP\left(\xi_N^*<x N^{1/s}\right)\to\exp\left[-\frac{\brbc}{s} x^{-s}\right]. $$
\end{theorem}
Theorem \ref{ThRec} shows that the fact that traps are Poisson distributed is useful even for $s>2.$
\begin{corollary}
\label{CrRec}
If $0<s<1$ then as $N\to\infty$
$$ \lim\sup \frac{\xi_N^*}{T_N}>0 $$
almost surely.
\end{corollary}
\begin{remark} Corollary \ref{CrRec} is a minor modification of the result of \cite{GS}.
Namely, in \cite{GS} the authors consider not all visits to site $n$
but only visits before $\tilde{T}_N.$ By Lemma \ref{LmS-T} this
difference is not essential since most visits occur before $\tilde{T}_N$.
\end{remark}

\section{Preliminaries.}
\label{ScPrelim}
\subsection{Poisson process.}\label{SSPP}
The proofs of the facts listed below can be found in monographs \cite{Resnick, SamT}.

Let $(X, \mu)$ be a measure space.
Recall that a Poisson process is a point process on $X$ such that \newline (a)
if $A\subset X$, $\mu(A)$ is finite, and $N(A)$ is the number of points in $A$
then $N(A)$ has a Poisson distribution with parameter $\mu(A)$;\newline (b)
if $A_1, A_2\dots A_k$ are disjoint subsets of $X$ then $N(A_1), N(A_2)\dots N(A_k)$
are mutually independent.

If $X\subset \reals^d$ and $\mu$ has a density $f$ with respect to the Lebesgue measure we
say that $f$ is the intensity of the Poisson process.
\begin{lemma}
\label{LmPT}
(a) If $\{\Theta_j\}$ is a Poisson process on $X$
and $\psi: X\to \tX$ is a measurable map
then $\tTheta_j=\psi(\Theta_j)$ is a Poisson process.
If $X=\tX=\reals$ and $\psi$ is invertible then the intensity of $\tTheta$ is
\begin{equation}
\label{ChangeInt}
\tf(\theta)=f(\psi^{-1}(\theta)) \left|\frac{d\psi}{d\theta}\right|^{-1}.
\end{equation}

(b) Let $(\Theta_j, \Gamma_j)$ be a point process on $X\times Z$ such that
$\{\Theta_j\}$ is a Poisson process on $X$
and $\{\Gamma_j\}$ are $Z$-valued random variables
which are i.i.d. and independent of $\{\Theta_k\}$.
Then $(\Theta_j, \Gamma_j)$ is a Poisson process on $X\times Z.$

(c) If in (b) $X=Z=\reals$ then
$\tTheta=\{\Gamma_j \Theta_j\}$ is a Poisson process. Its intensity is
$$ \tf(\theta)=
E_\Gamma \left(f\left(\frac{\theta}{\Gamma}\right)\frac{1}{\Gamma}\right) .$$
\end{lemma}

\begin{lemma}
\label{PChar}
Let $\Theta$ be Poisson process on $X,$ $\psi:X\to\reals$ a measurable function with
$\int |\psi(\theta)| d\mu(\theta)<\infty$ then
$$ V=\sum_j \psi(\theta_j)$$
is finite with probability 1, the characteristic function of $V$ is given by
\begin{equation}\label{char}
E_{\Theta}(\exp(ivV))=\exp\left[\int\left(e^{iv\psi(\theta)}-1\right)d\mu(\theta)\right],
\end{equation}
and
\begin{equation}
\label{PExp1}
E_{\Theta}(V)=\int \psi(\theta) d\mu(\theta).
\end{equation}
If in addition to the above conditions  $\int \psi^2(\theta) d\mu(\theta)<\infty$ then
\begin{equation}
\label{PExp2}
\Var_{\Theta}(V)=\int \psi^2(\theta) d\mu(\theta)
\end{equation}
\end{lemma}
\begin{remark} Proofs of the statements listed in Lemmas \ref{LmPT} and \ref{PChar} can be found in \cite{Resnick}.
\end{remark}
\begin{lemma}
\label{LmPSt}
(a) If $0<s<1$ and $\Theta_j$ is a Poisson process with intensity $\theta^{-(1+s)}$ then
$ \sum_j \Theta_j $ has a stable distribution of index $s.$

(b) If $1<s<2$ and $\Theta_j$ is a Poisson process with intensity $\theta^{-(1+s)}$ then
$$ \lim_{\delta\to 0}
\left[\left(\sum_{\delta<\Theta_j} \Theta_j\right)-\frac{1}{(s-1)\delta^{s-1}} \right]$$
has a stable distribution of index $s.$

(c) If $s=1$ and $\Theta_j$ is a Poisson process with intensity $\theta^{-2}$ then
$$ \lim_{\delta\to 0} \left[\left(\sum_{\delta<\Theta_j} \Theta_j\right)-|\ln\delta| \right]$$
has a stable distribution of index $1.$
\end{lemma}
\begin{remark} The proof of Lemma \ref{LmPSt} follows from a direct computation
of the characteristic function of the relevant sums in (a), (b), (c) using formula (\ref{char}).
We also note that the expressions under the limit sign in (b) and (c) are equal to
$\sum_{\delta<\Theta_j} \Theta_j-E_{\Theta}\left(\sum_{\delta<\Theta_j} \Theta_j\right)$. One
thus could say that the existence of the limit means that the series
$\sum_{j} (\Theta_j-E_{\Theta}(\Theta_j))$ converges. However, for this interpretation of
one has to introduce an ordering relation on the random sets $\{\Theta_j\}$ (see \cite{SamT}).
\end{remark}

\subsection{Backtracking.}
\begin{lemma}
\label{LmBack}
(\cite{GS}, Lemma 3.3) There exist $C>0, \beta<1$ such that
$$\bP(X \text{ visits } n \text { after } n+m)\leq C\beta^m.$$
\end{lemma}

\subsection{Occupation times. Recurrence relation.}
As before, let $\xi_n$ be the number of visits to the site $n$ and
$\rho_n=\EXP_\omega \xi_n$. Observe that $\xi_n$  has geometric distribution
with parameter $1/\rho_n$.
\begin{lemma}
\label{Lm3.1} If $X_0=0$ then for $n\ge0$
\begin{equation} \label{OTRR}
\rho_n=p_n^{-1}q_{n+1}\rho_{n+1}+p_n^{-1}=p_n^{-1}(1+\alpha_{n+1}+
\alpha_{n+1}\alpha_{n+2}+...),
\end{equation}
where $\alpha_j=\frac{q_j}{p_j}.$
\end{lemma}
\begin{proof}
Let $\eta_n^+$ and $\eta_n^-$  be the number of passages of the edge
$[n, n+1]$ in the forward, respectively, backward direction. Denote
 $\sigma_n^\pm=\EXP_\omega \eta_n^\pm.$
We have
$$\rho_n=\sum_j \Prob_\omega(X_j=n)\text{ and }
\sigma_n^+=\sum_j \Prob_\omega(X_j=n, X_{j+1}=n+1). $$ Thus
$\sigma_n^+=\rho_n p_n.$ Likewise $\sigma_n^-=\rho_{n+1} q_{n+1}.$
Since $\xi_n\to +\infty$ we have that $\eta_n^+-\eta_n^-=1$ for $n\ge0$. Hence
$$ \rho_n p_n-\rho_{n+1} q_{n+1}=1. $$
This implies the first relation in (\ref{OTRR}). The second one is
obtained by iterating the first one.
\end{proof}
For future references, we shall mention here several elementary but
 useful relations for $\rho_n$. We start with a direct corollary of (\ref{OTRR}):
\begin{equation}
\label{SManySteps}
\rho_{n-k} =p_{n-k}^{-1} \alpha_{n-k+1}\dots \alpha_{n-1}q_n\rho_n
+(1+\alpha_{n-k+1}+\dots+\alpha_{n-k+1}\dots \alpha_{n-1})p_{n-k}^{-1}.
\end{equation}
Set $\brc:=\eps_0^{-1}$, where $\eps_0$ is from condition (C). Then \eqref{SManySteps} implies that
\begin{equation}
\label{SManyStepsIneq-ty1}
\rho_{n-k}\le \brc A_{n,k}\rho_n+\brc B_{n,k},
\end{equation}
where $A_{n,k}:=\alpha_{n-1}\dots \alpha_{n-k+1}$, $B_{n,k}:=1+\alpha_{n-k+1}+\dots+\alpha_{n-k+1}\dots \alpha_{n-1}$. Next, \eqref{SManyStepsIneq-ty1} implies
\begin{equation}
\label{SManyStepsIneq-ty}
\rho_{n-k}\le \brc A_{n,k}\rho_n+\brc k\eps_0^{-k}.
\end{equation}
Note that $A_{n,k}$ and $\rho_n$ are independent random variables.

Next, we introduce
\begin{equation}\label{z2}
\begin{aligned}
z_n&:= 1+\alpha_{n+1} +\alpha_{n+1}\alpha_{n+2}+...+\alpha_{n+1}...\alpha_{n+m}+....\\
&=1+\alpha_{n+1} +\alpha_{n+1}\alpha_{n+2}+...+\alpha_{n+1}...\alpha_{n+m}z_{n+m}
\end{aligned}
\end{equation}
It is clear that $z_n=1+\alpha_{n+1}z_{n+1}$, where $\alpha_{n+1}$ and $z_{n+1}$
are independent random variables and the sequence $\{z_n\}_{-\infty<n<\infty}$
considered backward in time forms a Markov chain.
Obviously, $\rho_n=p_n^{-1}z_n$ is a function on the phase space of a
Markov chain $\{p_n, z_{n}\}$ (where $p_n$ and $z_{n}$ are independent).
Since $\bE(\ln\alpha)<0$ the series in
(\ref{OTRR}) and \eqref{z2} converge $\bP$-almost surely and the distributions of
$z_n$ and of $(p_n, z_{n})$ are the
stationary measures of the respective processes. The following heavy tail property of these
stationary measures plays a very important role in the sequel.
\begin{lemma}
\label{LmRen}
(\cite{K1}) There exist $c$ and $c^*>0$ such that
$$ \lim_{x\to+\infty} x^s \bP(z_n> x)=c,\ \
\lim_{x\to+\infty} x^s \bP(\rho_n> x)=c^*,$$
where $s>0$ satisfies $\bE(\alpha^s)=1$ (as in condition (B)).
\end{lemma}
Note that here the second relation is a simple corollary of the first one
because $\bP(\rho_n> x)=\bE[\bP(z_n> xp_n|p_n)]\sim \bE(c x^{-s}p^{-s})=cx^{-s}\bE(p^{-s})$.
We also see that $c^*=c\bE(p^{-s})$.
\begin{lemma}
\label{Lm2Clusters}
There exist $\eps_1>0, \eps_2>0, 0<\beta<1$ such that for any $\delta>0$ there are $N_\delta$
and $C=C_\delta>0$ such that for $N>N_\delta$ one has:\newline
(a) If $k\leq\eps_1 \ln N$ then
$$ \bP(\rho_n\geq \delta N^{1/s}, \rho_{n-k}\geq \delta N^{1/s})\leq
\frac{C\beta^k}{N}; $$
(b) If $k\geq\eps_1 \ln N$ then
$$ \bP(\rho_n\geq \delta N^{1/s}, \rho_{n-k}\geq \delta N^{1/s})\leq
C N^{-(\eps_2+1)}. $$
\end{lemma}

\begin{proof}
(a) It follows from \eqref{SManyStepsIneq-ty} that if $\eps_1$ is chosen so that
$-\eps_1\ln\eps_0\le\frac{1}{3s}$ and $N$ is sufficiently large then
$$
\rho_{n-k}
\le \brc\rho_n A_{n,k}+\brc\eps_1(\ln N)\eps_{0}^{-\eps_1\ln N}
\le \brc\rho_n A_{n,k}+\brc N^{\frac{1}{2s}}.$$
Next, there exist $\beta_1, \beta_2<1$ such that
\begin{equation}\label{beta}
\bP(\alpha_{n-1} \dots \alpha_{n-k}\geq \beta_1^k)\leq \beta_2^k.
\end{equation}
Indeed, if $0<h<\min(1,s)$ and $\beta_1$ is such that $\bE(\alpha^h)<\beta_1^h<1$ then it
follows from the Markov's inequality that
\begin{equation}\label{beta1,2}\bP( \alpha_{n-1} \dots \alpha_{n-k}\geq \beta_1^k)
\leq \frac{(\bE(\alpha^h))^k}{\beta_1^{hk}}\equiv \beta_2^k.
\end{equation}
We can now choose $N_\delta$ so that for $N>N_\delta$ we shall have
$$ \begin{aligned}
\bP(\rho_n\geq \delta N^{1/s}, \rho_{n-k}\geq \delta N^{1/s})&\leq
\bP(\rho_n\geq \delta N^{1/s}, \brc\rho_{n}A_{n,k}+\brc N^{\frac{1}{2s}}\geq \delta N^{1/s})\\
&\leq\bP(\rho_n\geq \delta N^{1/s}, \brc\rho_{n}A_{n,k}\geq \frac{\delta}{2} N^{1/s})\\
\end{aligned}
$$
Finally, the right hand side in the above inequality is estimated as follows:
$$
\begin{aligned}
&\bP(\rho_n\geq \delta N^{1/s},\, \brc\rho_{n}A_{n,k}\geq \frac{\delta}{2} N^{1/s})\\
&=\bP(\rho_n\geq \delta N^{1/s},\, \brc\rho_{n}A_{n,k}\geq \frac{\delta}{2} N^{1/s},\, A_{n,k}\leq\beta_1^k )\\
&+\bP(\rho_n\geq \delta N^{1/s},\, \brc\rho_{n}A_{n,k}\geq \frac{\delta}{2} N^{1/s},\, A_{n,k}>\beta_1^k) \\
&\le\bP\left(\rho_n\geq \frac{\beta_1^{-k} \delta N^{1/s}}{2\brc}\right)+
\bP(\rho_n>\delta N^{1/s} \text{ and }A_{n,k} > \beta_1^k) \leq \Const \frac{\beta_1^{ks}+\beta_2^k}{N},
\end{aligned}$$
where the last step makes use of Lemma \ref{LmRen} (hence the dependence of the $\Const$ on $\delta$) and of independence of $\rho_n$ and $A_{n,k}$.

(b) For any $\eps_3>0$ we can write
\begin{equation}\label{1}
\begin{aligned}
&\bP(\rho_n\geq \delta N^{1/s},\, \rho_{n-k}\geq \delta N^{1/s})\\
&=\bP( \delta N^{1/s}\leq \rho_n\leq \delta N^{\frac{1+\eps_3}{s}},\, \rho_{n-k}\geq \delta N^{1/s})
+\bP(\rho_n> \delta N^{\frac{1+\eps_3}{s}},\, \rho_{n-k}\geq \delta N^{1/s}) \\
&\le\bP( \delta N^{1/s}\leq \rho_n\leq \delta N^{\frac{1+\eps_3}{s}},\, \rho_{n-k}\geq \delta N^{1/s})
+\bP(\rho_n> \delta N^{\frac{1+\eps_3}{s}}) \\
&\leq \frac{\bar\bar{c}}{N^{1+\eps_3}}+\bP( \delta N^{1/s}\leq \rho_n\leq \delta N^{\frac{1+\eps_3}{s}},\, \rho_{n-k}\geq \delta N^{1/s}),
\end{aligned}
\end{equation}
where the last step follows from Lemma \ref{LmRen}. We use \eqref{SManyStepsIneq-ty1} to estimate the last term in \eqref{1}:
\begin{equation}\label{2}
\begin{aligned}
&\bP( \delta N^{1/s}\leq \rho_n\leq \delta N^{\frac{1+\eps_3}{s}},\, \rho_{n-k}\geq \delta N^{1/s})\\
&\le\bP( \delta N^{1/s}\leq \rho_n\leq \delta N^{\frac{1+\eps_3}{s}},\, \brc A_{n,k}\rho_n+\brc B_{n,k}\geq \delta N^{1/s})\\
&\le\bP( \delta N^{1/s}\leq \rho_n\leq \delta N^{\frac{1+\eps_3}{s}},\, \brc A_{n,k}
\delta N^{\frac{1+\eps_3}{s}}+\brc B_{n,k}\geq \delta N^{1/s})\\
&=\bP( \delta N^{1/s}\leq \rho_n\leq \delta N^{\frac{1+\eps_3}{s}})\,\bP(\brc A_{n,k}
\delta N^{\frac{1+\eps_3}{s}}+\brc B_{n,k}\geq \delta N^{1/s}),
\end{aligned}
\end{equation}
where the last step is due to the independence of $\rho_n$ and $( A_{n,k}, B_{n,k})$.
Next, let $1>h>0$ be such that $\brbeta=\bE(\alpha^h)<1$, then
$\bE(B_{n,k}^h)\le (1-\brbeta)^{-1}$. By Markov's inequality
\begin{equation}\label{3}
\begin{aligned}
&\bP(\brc A_{n,k}
\delta N^{\frac{1+\eps_3}{s}}+\brc B_{n,k}\geq \delta N^{1/s})\le
\brc^h\frac{\bE(\delta^hN^{\frac{1+\eps_3}{s}h}A_{n,k}^h+B_{n,k}^h)}{\delta^h N^{h/s}}\\
&\le\bar{\bar{c}} N^{\frac{\eps_3h}{s}}\brbeta^k+\bar{\bar{c}} N^{\frac{-h}{s}}.
\end{aligned}
\end{equation}
Since $k\ge\eps_1\ln N$, we have that $N^{\frac{\eps_3h}{s}}\brbeta^k\le N^{\frac{\eps_3h}{s}+\eps_1\ln\brbeta}=N^{-\bar{\varepsilon}}$ (with $\eps_3$ sufficiently small so that
to make $\bar{\varepsilon}$ strictly positive). Finally, it follows from Lemma \ref{LmRen}, \eqref{2} and \eqref{3}
that
\begin{equation}\label{4}
\bP( \delta N^{1/s}\leq \rho_n\leq \delta N^{\frac{1+\eps_3}{s}},\, \rho_{n-k}\geq \delta N^{1/s})\\
\le\Const N^{-1 - \min(\bar{\varepsilon},\, h/s)}.
\end{equation}
The proof of (b) now follows from \eqref{4} and \eqref{1}.
\end{proof}
Next,  we need the fact that $\rho_n$ is exponentially mixing.
To prove this we use \eqref{SManySteps}.
Assumptions (A) and (C) imply that
there exist $\beta_1, \beta_2<1$ such that
$$ \bP\left(\max_{k>L} \alpha_n \alpha_{n-1}\dots \alpha_{n-k-1}\geq \beta_1^{L}\right)
\leq \beta_2^{L} . $$
Therefore for typical realization of $\alpha$ the dependence of $\rho_{n-k}$ on $\rho_n$ decays exponentially.
We formulate this statement as follows. Given a $\hrho_n$ define for $k>0$
\begin{equation}
\label{SManySteps1} \hrho_{n-k} =p_{n-k}^{-1}\hrho_nq_n \alpha_{n-1}\dots \alpha_{n-k+1}
+(\alpha_{n-1}\dots \alpha_{n-k+1} +\dots+1)p_{n-k}^{-1}.
\end{equation}
We are mainly interested in the case when the difference between $\hrho_n$ and $\rho_n$
is large. More specifically we assume that $\hrho_n^h\gg \bE(\rho_n^h)$,
where $0<h<\min(1,s)$ is as in \eqref{beta1,2}. Then the following holds.
\begin{lemma}
\label{LmCouple}
Let $\hrho_{n-k}$ be defined by \eqref{SManySteps1} and $\rho_{n}$ be
the stationary sequence satisfying \eqref{SManySteps}.
Then there exist $K>0$ and $\beta_1,\,\beta_3<1$ such that for $k>K\ln\hrho_n$
$$ \bP\left(\left|\rho_{n-k}-\hrho_{n-k}\right| \geq \beta_1^k\right)
\leq \beta_3^k. $$
\end{lemma}
\begin{proof} It follows from \eqref{SManySteps} and \eqref{SManySteps1} that
$$\left|\rho_{n-k}-\hrho_{n-k}\right|\le \brc A_{n,k}\,\left|\rho_{n}-\hrho_{n}\right|.$$
Consider the same $0<h<1$, $\beta_1$, and $\beta_2$ as in \eqref{beta}, \eqref{beta1,2}
and set $\beta_3=(1+\beta_2)/2$. Then
$$ \bP\left(\left|\rho_{n-k}-\hrho_{n-k}\right| \geq \beta_1^k\right)
\leq \bP\left(\brc A_{n,k}\,\left|\rho_{n}-\hrho_{n}\right| \geq \beta_1^k\right)
\leq \beta_2^k[\bE(\rho_n^{h})+\hrho_{n}^h]\le \beta_3^k. $$
Here the first inequality is obvious. The second one is due to the Markov
inequality, to \eqref{beta},
and to the independence of $\rho_n$ and $A_{n,k}$. Finally, one easily checks that the third one holds for $k>K\ln\rho_n$, where $K:=2h/\ln(0.5+0.\beta_2^{-1} )+1$ (this where the condition  $\hrho_n\gg \bE(\rho_n^h)$ is used).
\end{proof}

\subsection{Occupation times. Correlations.}
The proofs of Lemmas \ref{LmCorShort} and \ref{LmCorLong}
will make use of several elementary equalities and inequalities
concerned with a Markov chain $Y=\{Y_t,\ t\ge0\}$ with a phase space
of 3 sites and transition matrix
\begin{equation}
\label{MatrMC}
\left(\begin{array}{ccc}
\brp  & \brq  & 0 \cr
\brrq & \brrp & \eps \cr
0      &  0      & 1    \cr
\end{array}\right).
\end{equation}
Namely, let $\breta$ and $\brreta$ be the total numbers of visit to the first and the second
site respectively. Set $U_1=E(\breta | Y_0=1)$,
$U_2=E(\breta | Y_0=2)$, $V_1=E(\brreta | Y_0=1)$, $V_2=E(\brreta | Y_0=2)$.
It follows easily from the standard first step analysis that
\begin{equation}
\label{ExpMC}
U_1=\frac{\eps+\brrq}{\eps\brq},\ \ U_2=\frac{\brrq}{\eps\brq},\ \
V_1=V_2=\frac{1}{\eps}.
\end{equation}
Next, set $W_i=E(\breta\brreta|Y_0=i)$, where $i=1,\, 2$. Once again, by the
first step analysis, one easily obtains that
\begin{equation}
\label{ExpProdMC}
W_1=\brp W_1+\brq W_2+ V_1,\ \ W_2=\brrq W_1+\brrp W_2+ U_2.
\end{equation}
Solving \eqref{ExpProdMC} gives
\begin{equation}
\label{ExpProdMC1}
W_1= V_1(U_1+U_2),\ \ W_2=U_2(V_1+V_2)
\end{equation}
and hence
\begin{equation}
\label{CovMC}
\Cov(\breta, \brreta|Y_0=1)=\Cov(\breta, \brreta|Y_0=2)=V_1U_2.
\end{equation}
It is a standard fact that $\breta$ conditioned on $Y_0=1$ has geometric distribution
whose parameter is thus $U_1^{-1}$. If our Markov chain starts
from 1 it must visit 2 before being absorbed by 3. Hence the distribution of
$\brreta$ conditioned on $Y_0=1$ is the same as the distribution of
$\brreta$ conditioned on $Y_0=2$ and is geometric with parameter $V_2^{-1}=\eps$.
We therefor have that $\Var(\breta|Y_0=1)=U_1^{2}-U_1$ and $\Var(\brreta|Y_0=1)=V_2^{2}-V_2$.
We can now compute the correlation coefficient of $\breta$ and $\brreta$ which,
taking into account \eqref{ExpMC}, can be presented as follows:
\begin{equation}
\label{CorMC}
\Corr(\breta, \brreta|Y_0=1)=\frac{V_1U_2}{\sqrt{(U_1^{2}-U_1)(V_2^{2}-V_2)}}=
\frac{\brrq}{\brrq+\eps}(1-U_1^{-1})^{-\frac{1}{2}}(1-V_2^{-1})^{-\frac{1}{2}}.
\end{equation}
This formula implies lower and upper bounds for correlations in two
different regimes: (a) when $\brrq/\eps\to0$ and (b) when $\eps\to 0$ while $\brq,\ \brrq$ remain
separated from $0$. Here is the precise statement we need.
\begin{lemma}
\label{LmCorEst} (a) Suppose that $U_1\ge1+c$, $V_2\ge1+c$, where $c>0$.
Then
\begin{equation}\label{Cor1MC}
\Corr(\breta, \brreta|Y_0=1)\le \Const \,\frac{\brrq}{\eps}\equiv\Const \,{\brrq}V_1.
\end{equation}
(b) If  ${\brrq}\ge c$ and ${\brq}\ge c$ for some $c>0$ then
for $\eps$ small enough, or,
equivalently, $U_1$ large enough
\begin{equation}
\label{CorEpsUMC}
\Corr(\breta, \brreta|Y_0=1)\ge 1-\frac{\eps}{c},\ \ \
\Corr(\breta, \brreta|Y_0=1)\ge 1-\frac{1}{c U_1}.
\end{equation}
\end{lemma}
\begin{proof} (a) Inequality \eqref{Cor1MC} is an immediate corollary of \eqref{CorMC}.

(b) \eqref{CorMC} can be written as
$$
\Corr(\breta, \brreta|Y_0=1)=
\frac{\brrq}{\brrq+\eps}(1-\frac{\eps\brq}{\brrq+\eps})^{-\frac{1}{2}}(1-\eps)^{-\frac{1}{2}}.
$$
If $\frac{\eps}{\brrq}<1$ then it follows from here that
\begin{equation}
\label{CorEpsMC}
\Corr(\breta, \brreta|Y_0=1)=1- (1-\frac{\brq+\brrq}{2})\frac{\eps}{\brrq}
+\mathcal{O}(\left(\frac{\eps}{\brrq}\right)^2).
\end{equation}
Due to \eqref{ExpMC} and conditions of the Lemma we have  $\eps=\frac{\brrq}{\brq}\left(U_1^{-1} +\mathcal{O}(U_1^{-2})\right)$ and hence
\begin{equation}
\label{CorU1MC}
\Corr(\breta, \brreta|Y_0=1)=1- (1-\frac{\brq+\brrq}{2})\frac{1}{\brq U_1} +\mathcal{O}(U_1^{-2}).
\end{equation}
\eqref{CorEpsUMC} is now a simple corollary of \eqref{CorEpsMC} and \eqref{CorU1MC}.
\end{proof}

\begin{lemma}
\label{LmCorShort}There is a $C>0$ such that for $\mathbf{P}$-almost all $\omega$ and $n\ge0$
\begin{equation}
\label{CorrhoMC}
\Corr_\omega(\xi_n, \xi_{n+1})\geq 1-\frac{C}{{\rho_n}}.
\end{equation}
\end{lemma}
\begin{proof} Let $\omega$ be such that the random walk $X$ runs away to $+\infty$
with $\mathbb{P}_\omega$ probability 1 (which is the case for
$\mathbf{P}$-almost all $\omega$). For a given $n\ge0$ consider a Markov
chain $Y=\{Y_t,\ t\ge0\}$,
with the state space $\{n,\ n+1,\ as\}$, where $n,\ n+1$ are sites on
$\mathbb{Z}$ and $as$ is an absorbing state. Let $k_0< k_1<...<k_\tau$ be the sequence of all
moments such that $X_{k_j}\in\{n,n+1\}$; we set $Y_t=X_{k_t}$ if $t\le\tau$
and $Y_t=as$ if $t>\tau$.
It easy to see that the transition matrix of $Y$ is as in \eqref{MatrMC} with
transition probabilities given by
$$ \brp=q_n,\ \ \brq= p_{n},\ \ \brrq=q_{n+1},$$
$$\brrp=\mathbb{P}_\omega
\{\hbox{$X_k$ starting from $n+1$ returns to $n+1$ before visiting $n$}\},$$
$$  \eps=\mathbb{P}_\omega
\{\hbox{$X_t$ starting from $n+1$ never returns to $n+1$}\}.$$
Also, in this context, $\breta=\xi_n$, $\brreta=\xi_{n+1}$ and hence $V_1=\rho_n$.
Next, $\brq,\ \brrq$ are separated from $0$
because of condition (C) from Section \ref{Intro}.
All conditions of Lemma \ref{LmCorEst} are thus satisfied and hence, for
$\rho_n$s which are sufficiently large, \eqref{CorrhoMC} follows from \eqref{CorEpsUMC}.
\end{proof}
\begin{lemma}
\label{LmCorLong}
(a) There exist sets $\Omega_N,$ $K>0$
such that
$ \bP(\Omega_N^c)\leq N^{-100} $
and if $\omega\in \Omega_N$ then for all $0\le n_1,\,n_2\le N$
such that $n_2>n_1+K\ln N$ we have
$$ \Corr_\omega(\xi_{n_1}, \xi_{n_2})\leq N^{-100} . $$
(b) If $K$ is sufficiently large then
for each $N$ there exist random variables $\{\brxi_n\}_{n=0}^N$ such that
for each $\omega\in \Omega_N$
for any sequence $0\le n_1<n_2\dots <n_k\le N$ such that $n_{j+1}>n_j+K\ln N,$
the variables
$\{\brxi_{n_{j}}\}_{j=0}^k$ are mutually independent and
\begin{equation}
\label{AlmostSame}
\bP(\brxi_{n}=\xi_{n} \text{ for } n=0,\dots, N)\geq 1-\frac{C}{N^{100}}.
\end{equation}
\end{lemma}

\begin{proof}
(a) Consider a Markov chain $Y$ which is defined as in the proof of
Lemma \ref{LmCorShort} with the difference that its state space is
$\{n_1,n_2,as\}$ and that $\breta=\xi_{n_1}$, $\brreta=\xi_{n_2}$. Then by \eqref{Cor1MC}
$$
\Corr_\omega(\xi_{n_1}, \xi_{n_2})\le \Const \,{\brrq}\rho_{n_2}.
$$
But, by Lemma \ref{LmRen},  $\rho_{n}\le N^{\frac{103}{s}}$ except for the
set of measure $\cO(N^{-103}).$
Now Lemma \ref{LmBack} guarantees that we can choose $K$ so that if
the sites are separated by $K\ln N$ then $\brrq<N^{-(101+103/s)}$ except for the set of measure
$\cO(N^{-103}).$ This proves (a) for fixed $n_1,n_2$ on a set of measure $\ge 1-\cO(N^{-103})$
which in turn implies the wanted result.

(b) Let $\brxi_n$ be the number of visits to the site $n$
before the first visit to $n+\frac{K\ln N}{2}.$ It follows from this definition that
$\{\brxi_{n_{j}}\}_{j=0}^k$ are mutually independent. Next,
$$\bP(\brxi_{n}=\xi_{n})\le \bP(X \text{ visits } n \text{ after } n+0.5 K\ln N)$$
Now \eqref{AlmostSame} follows from Lemma \ref{LmBack}.
\end{proof}

\section{Proof of Theorem \ref{ThMain}.}
\label{ScProof}
Our goal is to show that the main contribution to
$T_N$ comes from the terms where $\rho_n$ is large. However, the set where $\rho_n$ is
large has an additional structure. Namely, if $\rho_n$ is large the same is true for
$\rho_{n\pm 1}$ and more generally for $\rho_{n_1}$  and $\rho_{n_2}$ when
$n_1$ and $n_2$ are in a sense close to $n$; this implies that the corresponding $\xi_{n_1}$ and $\xi_{n_2}$ are strongly correlated. But if $n_1$ and $n_2$ are far apart then $\rho_{n_1}$  and $\rho_{n_2}$, and also $\xi_{n_1}$ and $\xi_{n_2}$, are almost independent.
In the arguments below we need to take care about this additional structure.

But first we show that terms where $\rho_n<\delta N^{1/s}$ can be neglected.

\begin{lemma}
\label{LmLow}
Let $\delta> 0$. Then there is $N_\delta$ (which depends also on $s$) such that
for $N>N_\delta$ the following holds:

(a) If $0<s<1$ then
$$ \bE\left(\sum_{\rho_n<\delta N^{1/s}} \xi_n\right)\le \Const N^{1/s}\delta^{1-s}. $$

(b) If $1<s<2$ then
there is a set $\tOmega_{N,\delta}$ such that $\bP(\tilde{\Omega}_{N,\delta}^c)\leq N^{-100}$ and
$$ \bE\left(1_{\tOmega_{N, \delta}}
\mathbb{E}_\omega\left(\sum_{\rho_n<\delta N^{1/s}} (\xi_n-\rho_n)\right)^2\right)
\le \Const N^{2/s}\delta^{2-s}. $$

(c) If $0<s<1$ then
$$ \bE\left(\sum_{\rho_n<\delta N^{1/s}} \rho_n\right)\le \Const N^{1/s}\delta^{1-s}. $$

(d) If $1<s<2$ then
$$ \bE\left(\left(\sum_{\rho_n<\delta N^{1/s}} (\rho_n-\bE(\rho))\right)^2\right)\le
\Const N^{2/s}\delta^{2-s}. $$

(e) If $s=1$ then given $\kappa<1$ there is a set $\tOmega_{N, \delta}$ such that
$\bP(\tilde{\Omega}_{N,\delta}^c)\leq N^{-100}$ and
\begin{equation}
\label{LLe1}
\bE\left(1_{\tOmega_{N,\delta}} \Var_\omega \left(\sum_{\rho_n<\delta N} (\xi_n-\rho_n)\right)^\kappa\right)
\le \Const N^{2\kappa} \delta^{2\kappa},
\end{equation}
\begin{equation}
\label{LLe2}
\bE\left(\left(\sum_{\rho_n<\delta N} \left(\rho_n-\bE\left(\rho I_{\rho<\delta N}\right)\right)
\right)^{2\kappa}\right)
\le \Const N^{2\kappa} \delta^{2\kappa}.
\end{equation}
\end{lemma}

\begin{proof}
(a) Denote $Y_\delta=\sum_{\rho_n<\delta N^{1/s}} \xi_n.$ Then
$$ \bE(Y_\delta)=N \bE(\rho I_{\rho<\delta N^{1/s}}). $$
By Lemma \ref{LmRen} this expectation is bounded by $\Const N^{1/s} \delta^{1-s}$
proving our claim.

(b) Denote $\tY_\delta=\sum_{\rho_n<\delta N^{1/s}} (\xi_n-\rho_n).$ Then
$\mathbb{E}_\omega(\tY_n)=0$ and so it suffices to show that $\Var_\omega(\tY_\delta)=o(N^{2/s})$
except for a set of small probability.
Due to Lemma \ref{LmCorLong} for most $\omega$s we have
\begin{equation}
\begin{aligned}
\label{VAR}
\Var_\omega(\tY_\delta)
&=\left|o(1)+\sum_{n_2-K\ln N<n_1<n_2} 2\Cov_\omega\left(\xi_{n_1} \xi_{n_2}\right)
+ \sum_{n} \Var_\omega\left(\xi_{n} \right)\right|\\
&\leq
1+\Const\sum_{n_2-K\ln N<n_1\leq n_2} \rho_{n_1} \rho_{n_2} \\
\end{aligned}
\end{equation}
where the summation is over pairs with $\rho_{n_i}<\delta N^{1/s}$. The last step uses
Cauchy-Schwartz inequality and the fact that $\xi_n$ has geometric distribution, namely
$\left|\Cov_\omega\left(\xi_{n_1} \xi_{n_2}\right)\right| \leq \sqrt{\Var_\omega\left(\xi_{n_1}\right) \Var_\omega\left(\xi_{n_2}\right)}\leq \rho_{n_1} \rho_{n_2}$.

Next, we estimate the expectation of the last sum in \eqref{VAR}. Set $\chi_n=
I_{\rho_{n}<\delta N^{1/s}}$ and $\beta=\bE(\alpha)<1$; these concise notations will be used only
within this proof. Using \eqref{SManySteps} we can write
$$
\rho_{n-k}\rho_n =p_{n-k}^{-1}\rho_n^2q_n \alpha_{n-1}\dots \alpha_{n-k+1}
+(\alpha_{n-1}\dots \alpha_{n-k+1} +\dots+1)p_{n-k}^{-1}\rho_n.
$$
Since $\rho_n$ and $\{\alpha_j,\ j<n\}$ are independent we obtain
$$
\bE\left(\rho_{n-k}\rho_{n} \chi_{n}\right) \le
\Const\left[\beta^{k-1} \bE\left(\rho_{n}^2
\chi_{n}\right) +
 \bE\left(\rho_{n}\chi_{n}\right)\sum_{j=0}^{k-2} \beta^j\right]
$$
Thus
\begin{equation}
\label{SumCov}
 \bE\left(\sum_{k=0}^{K\ln N}(\rho_{n-k}\rho_{n} \chi_{n})\right)
\leq \Const \left[\bE\left(\rho_{n}^2\chi_n\right)+\ln N \bE\left(\rho_{n}\chi_n\right)\right].
\end{equation}
Hence
$$\bE\left(\sum_{n_2-K\ln N<n_1<n_2} \rho_{n_1} \rho_{n_2}\chi_{n_1}\chi_{n_2}\right)
\leq \bE\left(\sum_{n_2-K\ln N<n_1<n_2} \rho_{n_1} \rho_{n_2}\chi_{n_2}\right)
$$
$$
\leq\Const\sum_{n_2} \left[ \bE\left(\left(\rho_{n_2}\right)^2 \chi_{n_2}\right)+
 \ln N\bE(\rho_{n_2})\right]  \leq\Const \delta^{2-s} N^{2/s}. \
$$

(c) The proof of (c) is the same as proof of (a).

(d) We assume first that
\begin{equation}
\label{SoftDelta}
\nu([\delta N^{1/s} -N^{-100}, \delta N^{1/s}+N^{-100}])\leq N^{-50}.
\end{equation}

In view of Lemma \ref{LmRen} it is enough to to compute the variance of
$$\bsigma=\sum_{\rho_n<\delta N^{1/s}} (\rho_n-\bE(\rho_nI_{\rho_n<\delta N^{1/s}})) $$

Lemma \ref{LmCouple} shows that if \eqref{SoftDelta} holds then for $|n_2-n_1|\geq \tK \ln N$
$$ \Cov \left(\rho_{n_1}I_{\rho_{n_1}<\delta N^{1/s}},
\rho_{n_2}I_{\rho_{n_2}<\delta N^{1/s}}\right)<\frac{1}{N^3} $$
provided that $\tK$ is large enough.
Hence
$$\left|\Var(\bsigma)\right|\leq 1+\left|\sum_{|n_1-n_2|<\tK\ln N}
\Cov\left(\rho_{n_1}I_{\rho_{n_1}<\delta N^{1/s}},
\rho_{n_2}I_{\rho_{n_2}<\delta N^{1/s}}\right)\right| . $$
provided that $\tK$ is sufficiently large.
The estimate of the last sum is exactly the same as in part (b). This completes the
proof of part (d) in the case when \eqref{SoftDelta} holds.
In case \eqref{SoftDelta} fails
we can repeat the computation below with $\delta$ replaced by $\delta'$ and $\delta''$ where
$\delta', \delta''$ satisfy \eqref{SoftDelta} and such that
$$\delta<\delta'<\delta+N^{-(50+1/s)}, \quad \delta-N^{-(50+1/s)}<\delta''< \delta. $$
The bound for $\delta'$ will allow us to estimate the sum of part (d) from above and
the bound for $\delta''$ will allow us to estimate the sum of part (d) from below.

(e) We prove \eqref{LLe1}, \eqref{LLe2} is similar.
In view of \eqref{VAR} it suffices to estimate
$$\bE\left(\sum_{n_1\leq n_2<n_1+K\ln N, \rho_{n_1}<\delta N, \rho_{n_2}<\delta N}
\rho_{n_1} \rho_{n_2}\right)^\kappa $$
We have
$$\bE\left(\sum_{n_1\leq n_2<n_1+K\ln N, \rho_{n_1}<\delta N, \rho_{n_2}<\delta N}
\rho_{n_1} \rho_{n_2}\right)^\kappa $$
$$\leq \sum_{n_1\leq n_2<n_1+K\ln N}
\bE\left(\left(\rho_{n_1} \rho_{n_2}\right)^\kappa
I_{\rho_{n_1}<\delta N} I_{\rho_{n_2}<\delta N}\right).$$
Using that $\bE(\alpha^\kappa)<1$ we can proceed as in part (b) to estimate the last sum by
$$C \sum_n \bE\left(\left(\rho_n\right)^\kappa I_{\rho_n<\delta N}\right)=
\tC N^{2\kappa} \delta^{2\kappa}. $$
\end{proof}

Lemma \ref{LmLow} allows us to concentrate on sites where $\rho_n\ge\delta N^{1/s}. $
In view of Lemma \ref{LmRen} for each fixed $\delta$ we expect to have finitely many such points on $[0, N]$
(namely the expected number of points is $\cO(\delta^{-s})$).

\begin{definition} Let $M=M_N:=\ln\ln N.$
We shall say that $n$ is a {\it massive} site if $\rho_{n}\geq \delta
N^{1/s}$.
A site $n\in[0,N-1]$ is {\it marked} if it is massive
and $\rho_{n+j}<\delta N^{1/s} $ for
$1\leq j\leq M.$
For $n$ marked the interval $[n-M, n]$ is called the {\it cluster} associated to $n.$
\end{definition}
It may happen that
not all massive sites belong to one of the clusters.
This situation is controlled by the following
\begin{lemma}\label{non-cluster}
\begin{equation}
\label{Eq4Peaks}\bP\left(\rho_{n}\geq \delta
N^{1/s} \hbox{ and $n$ is not in a cluster }  \right)\le \Const \frac{\beta^M}{N} .
\end{equation}
\end{lemma}
\begin{proof} Suppose that  $n$ is a massive point which is not in a cluster.
Then consider all massive points $n_i$ such that $n<n_1<...<n_k<n+M$. Note that
such points exist because otherwise $n$ would have been a marked point.
Let now $n^*>n_k$ be the nearest to $n_k$ massive point. Then by construction  $n^*\ge n+M$.
Also $n^*\le n+2M$ because otherwise $n_k$ would have been a marked point and
$n$ would belong to the $n_k$-cluster. Hence the event
$$\{n \hbox{ is massive and not in a cluster}\}
\subset\bigcup_{n'\in[n+M, n+2M]}\{\rho_{n}\geq \delta
N^{1/s},\,\rho_{n'}\geq \delta
N^{1/s}\}.$$
By Lemma \ref{Lm2Clusters}(a) we obtain
$$\begin{aligned}
&\bP\left(n \hbox{ is massive and not in a cluster} \right)\\
&\le \sum_{n'=n+M}^{n+2M}\bP\left(
\rho_{n}\geq \delta
N^{1/s},\, \rho_{n'}\geq \delta
N^{1/s}\right)\le\Const \frac{\beta^M}{N}
\end{aligned}
$$
which proves our statement.
\end{proof}
It is clear from the just presented proof that the event \newline
{\centerline{$\{\hbox{there is $n$ which is massive and not in a cluster}\}$}}
belongs to the set of environments
\begin{equation}
\label{Eq2Peaks}
\brOmega_N^\delta:=\{\exists n_1, n_2: M< n_2-n_1\leq 2M\text{ and }
\rho_{n_1}\geq \delta N^{1/s}, \rho_{n_2}\geq \delta N^{1/s}\}.
\end{equation}
Then again by Lemma \ref{Lm2Clusters}(a) we have that
\begin{equation}
\label{Eq3Peaks}\bP\left(\brOmega_N^\delta\right)\le\sum_{n=0}^{N-1}\sum_{n_1=n-2M}^{n-M}
\bP\{\rho_{n_1}\geq \delta N^{1/s},\, \rho_{n}\geq \delta
N^{1/s}\}\le \Const\beta^{M}.
\end{equation}
Obviously $\bP(\brOmega_N^\delta)\to 0$ as $N\to\infty.$

It is clear from the definitions that $\bP(n\hbox{ is massive and in a cluster}) \ge  \bP(n\hbox{ is marked})$. The following lemma shows that in fact these quantities are of the same order of smallness.
\begin{lemma}\label{in-cluster}
\begin{equation}
\label{Eq5Peaks}\bP\left(\rho_{n}\geq \delta
N^{1/s} \hbox{ and $n$ is in a cluster }  \right)\le \Const
\bP\left(\hbox{ $n$ is marked }  \right).
\end{equation}
\end{lemma}
\begin{proof} The event
$$\{n \hbox{ is massive and in a cluster}\}
\subset\bigcup_{k=0}^{M}\{\rho_{n}\geq \delta
N^{1/s},\,n+k \hbox{ is marked}\}.$$
Since $\rho_{n}$  is a stationary sequence we have
$$\begin{aligned}
&\bP\{\rho_{n}\geq \delta
N^{1/s},\,n+k \hbox{ is marked}\}=\\
&\bP\{\rho_{n-k}\geq \delta
N^{1/s}\,|\,n \hbox{ is marked}\}\times
\bP\{n \hbox{ is marked}\}.
\end{aligned}$$
We shall now prove that $\bP\{\rho_{n-k}\geq \delta
N^{1/s}\,|\,n \hbox{ is marked}\}\le \Const \beta^k$, where $\beta=\bE(\alpha^h)$, $0<h<s$.
Since $M$ is growing very slowly we have for
$N\ge N_\delta$ that $M\eps_0^M\ll 0.5\delta N^{1/s}$.
Then \eqref{SManyStepsIneq-ty} implies that $\rho_{n-k}\le \brc A_{n,k}\rho_n+0.5\delta N^{1/s}$
and therefore
$$\bP\{\rho_{n-k}\geq \delta
N^{1/s}\,|\,n \hbox{ is marked}\}\le
\bP\{\brc A_{n,k}\rho_n>0.5\delta N^{1/s}|\,n \hbox{ is marked}\}.$$
For $n$ marked $\rho_{n+1}< \delta
N^{1/s}$ and hence $\rho_{n}< 2\eps_0^{-1}\delta N^{1/s}$. Since $A_{n,k}$
and $\{\rho_j\}_{j\ge n}$ are independent we have, with $C=2\brc\eps_0^{-1}$:
$$\begin{aligned}
&\bP\{\brc A_{n,k}\rho_n>0.5\delta N^{1/s}|\,n \hbox{ is marked}\}\\
&\le \bP\{ C A_{n,k}\delta N^{1/s}>0.5\delta N^{1/s}|\,n \hbox{ is marked}\}
=\bP\{ A_{n,k}>0.5C^{-1}\}\le \Const \beta^k.
\end{aligned}
$$
(Once again, last step is due to the Markov inequality.)
Finally we obtain
$$\begin{aligned}
&\bP\left(n \hbox{ is massive and is in a cluster} \right)\le \\
&\le\Const (\sum_{k=0}^{M} \beta^k)\times
\bP\{n \hbox{ is marked}\}\le\Const\,\bP\{n \hbox{ is marked}\}.
\end{aligned}
$$
\end{proof}

We shall now turn to the analysis of the properties of clusters.
The next lemma is the main technical result of the paper.
It will be proved in Section \ref{ScPVis}. We need one more
\begin{definition}
For each marked point $n,$ we set
\begin{equation}\label{Defabm}
a_n=\rho_n/\delta N^{1/s}, \quad
b_n=\frac{\Sigma_{j=0}^M \rho_{n-j}}{\rho_n} \text{ and }
m_n=\Sigma_{j=0}^M \rho_{n-j}=\delta N^{1/s} a_n b_n .
\end{equation}
We call $m_n$ the mass of the cluster.
\end{definition}
\begin{lemma}
\label{LmPVis} For a given $\delta>0$ the following holds:

(a) The point process $\{(\frac{n}{N}, a_n, b_n):\ n\ \hbox{ is marked }\}$ converges
as $N\to\infty$ to a point process
$\{(t_j, \ta_j, \tb_j)\}$ where $t_j$ form a Poisson process with a constant
intensity $\tc \delta^{-s}.$

(b) For a given (finite) collection $\{t_j\}$ the corresponding
collection $\{(\ta_j, \tb_j)\}$ consist of i.i.d. random variables
which are independent of
$\{t_j\}$ (except that both collections have the same cardinality).
The distributions of the pair $(\ta,\tb)$ does not depend on
$\delta$.

(c) Consequently\footnote{part (c) of Lemma \ref{LmPVis} follows from parts (a) - (b) and Lemma \ref{LmPT}.}
$\{\left(\frac{n}{N}, \frac{m_n}{N^{1/s}}\right)\}$ converges to a
Poisson process $\Lambda^\delta=\{(t_j, \Theta_j)\}$ on
$[0,1]\times[\delta, \infty).$
\end{lemma}
We claim that $\Lambda^\delta$ has a limit as $\delta\to 0$ in the
following sense. Let $\Phi$ be a continuous function whose support
is a compact set disjoint from the segment $\{\theta=0\}.$ Then
$$ \lim_{\delta\to 0} \bE_{\Lambda^\delta} \left(\sum_j \Phi(t_j, \Theta_j)\right) $$
exists. Indeed let $\tdelta<\delta.$ Consider again the converging sequence
$\{\left(\frac{n}{N}, \frac{m_n}{N^{1/s}}\right)\}$ corresponding to $\tdelta$.
We may have more clusters
corresponding to $\tdelta$ but for any fixed $b$ it is unlikely that one of those
new clusters will have mass greater than $bN^{1/s}$ since this would mean that
all points in that cluster would have
\begin{equation}
\label{LargeMass}
 \rho_{n}<\delta N^{1/s}, \text{ but }
\sum_{j=0}^M\rho_{n-j}>\frac{b N^{1/s}}{\delta},
\end{equation}
where the sum is over $j$ in the additional cluster.
But \eqref{LargeMass} is unlikely in view of parts (c) and (d) of Lemma \ref{LmLow}.

The second distinction between $\Lambda^\delta$ and $\Lambda^\tdelta$
is the following. Consider a $\delta$-cluster $\brC$ and a $\tdelta$-cluster
$\brrC$ intersecting it. Then $\brC$ and $\brrC$ are shifted with respect to each
other so they have different masses. However, with probability close to 1
the masses of all such pairs of clusters differ by a relatively small amount.
Indeed $\brrC\setminus\brC$  always contains only sites with $\rho_n<\delta N^{1/s}$
and $\brC\setminus\brrC$
is unlikely to contain sites where $\rho_n\geq \delta N^{1/s}$ since this is only
possible for $\omega\in\brOmega_N^\tdelta$ where $\brOmega_N^\tdelta$ is the set defined by
\eqref{Eq2Peaks}. On the other hand, from Lemma \ref{LmLow} we know that terms
with $\rho_n<\delta N^{1/s}$ are unlikely to make a large contribution.

Let $\Lambda=\lim_{\delta\to 0} \Lambda^\delta.$ Let $\{\Theta_j\}$ be the projection of
$\Lambda$ into the second coordinate. By Lemma \ref{LmPT}(a), $\{\Theta_j\}$ is a Poisson process.
\begin{lemma}
\label{LmPTheta}
There exists $\bc$ such that the intensity of $\{\Theta_j\}$ equals to
$\frac{\bc}{\theta^{1+s}}.$
\end{lemma}
\begin{proof}
For each $\kappa,$ we have $\Lambda=\lim_{\delta\to 0} \Lambda^{\kappa\delta}.$
$\Lambda^\delta$ depends on $\delta$ in two ways. First its intensity is proportional to
$\delta^{-s}.$ Second, $\Theta_j/\delta=\ta_j \tb_j.$
Recall that the distribution of  $\ta_j\tb_j$ is independent of $\delta.$
Therefore replacing $\delta$ by $\kappa\delta$ replaces $\Theta\to \kappa\Theta$ and multiplies the
intensity by $\kappa^{-s}.$ In other words
re-scaling $\{\Theta_j\}$  by $\kappa$ amounts to multiplying its intensity by $\kappa^{-s}.$
Now the result follows \eqref{ChangeInt}.
\end{proof}
We are now in a position to finish the proof of Theorem \ref{ThMain}.
We shall do that in the case $0<s<1$. In all other cases the proof is similar.

Present the time spent by the walk in $[0,N)$ as
\begin{equation}\label{TN}
T_N=\sum_{n=0}^{N-1} \xi_n=S_1+\ S_2+\ S_3,
\end{equation}
where
$$\begin{aligned}
S_1=&\sum_{n:\,\rho_n<\delta N^{1/s},\,n\,\not\in\,\text{any cluster}} \xi_n \\
S_2=&\sum_{n:\,\rho_n\ge\delta N^{1/s},\,n\text{ is not in a cluster}}\xi_n\\
S_3=&\sum_{n:\,n\text{ is in a cluster}}\xi_n.
\end{aligned}
$$
By Lemma \ref{LmLow}, (a) we have that $\bE(S_1)\le\Const N^{1/s}\delta^{1-s}$. Next by
\eqref{Eq2Peaks}, \eqref{Eq3Peaks} we have that
$$
\bP(S_2>0)\le \bP\left(\brOmega_N^\delta\right)\to0\text{ as } N\to\infty.
$$
We readily have that for $\omega\not\in\brOmega_N^\delta$
$$
\mathfrak{t}_N=N^{-\frac{1}{s}}S_3+N^{-\frac{1}{s}}S_1=N^{-\frac{1}{s}}S_3+R_N,
$$
where $R_N:=N^{-\frac{1}{s}}S_1$ and satisfies the requirements of (a), Theorem \ref{ThMain}.

Next, consider $S_3$ which comes from the sum over the clusters and
is the main contribution to $T_N$. Let us present it as follows:
$$
N^{-\frac{1}{s}}S_3=\sum_{n:\, n\text{ is marked}}N^{-\frac{1}{s}}\sum_{j=0}^M \xi_{n-j}.
$$
In turn
$$
\sum_{j=0}^M \xi_{n-j}=\sum_{j=1}^M
\left(\frac{\xi_{n-j}}{\rho_{n-j}}-\frac{\xi_{n}}{\rho_{n}}\right)\rho_{n-j}+
\frac{\xi_{n}}{\rho_{n}}\sum_{j=0}^M \rho_{n-j}
$$
Next, using Lemma \ref{LmCorShort} and the fact that $\xi_n$ is a geometric
random variable and therefore $\Var_\omega(\xi_n)=\rho_n^2-\rho_n$
one obtains
$$
\left\|\frac{\xi_{n-j}}{\rho_{n-j}}-\frac{\xi_{n}}{\rho_{n}}\right\|\le
\sum_{k=n-j}^{n-1} \left\|\frac{\xi_{k}}{\rho_{k}}-\frac{\xi_{k+1}}{\rho_{k+1}}\right\|
\le \Const \sum_{k=n-j}^{n-1} \frac{1}{\sqrt{\rho_k}}.
$$
Here and below $\|f\|:=\sqrt{\mathbb{E}_{\omega}(|f|^2)}$ with $f$ being a function
on the space of trajectories of the walk.

For $n-j$ belonging to a cluster, that is $(n-j)\in[n-M,n]$  we have that
$\rho_{n-j}\ge c\eps_0^M \rho_n\ge cN^{-\bar{\eps}}\rho_n$. (Remember that
if $\eps_1$ in Lemma \ref{Lm2Clusters} is small enough then $\bar{\eps}$ can be made very small
which is what we shall use in this proof.) Thus
$$
\left\|\sum_{j=1}^M
\left(\frac{\xi_{n-j}}{\rho_{n-j}}-\frac{\xi_{n}}{\rho_{n}}\right)\rho_{n-j}\right\|
\le \Const \frac{N^{\bar{\eps}/2}}{\sqrt{\rho_{n}}}\sum_{j=1}^M\rho_{n-j}
$$
If for $n$ marked we set
$$
\zeta_n=m_n^{-1}\sum_{j=1}^M
\left(\frac{\xi_{n-j}}{\rho_{n-j}}-\frac{\xi_{n}}{\rho_{n}}\right)\rho_{n-j}
$$
then $\|\zeta_n\|\le \Const \frac{N^{\bar{\eps}/2}}{\sqrt{\rho_{n}}}\to0$ as $N\to\infty$
and we have
$$\frac{\sum_{j=0}^M \xi_{n-j}}{N^{1/s}}=\left(\frac{\xi_n}{\rho_n}
+\zeta_n\right) \frac{m_n}{N^{1/s}},$$
Next $\xi_n/\rho_n$ is
asymptotically exponential with mean 1 since $\xi_n$ is geometric
with parameter $1/\rho_n.$ Also by Lemma \ref{LmCorLong}
$\{\frac{\xi_n}{\rho_n}\}_{n \text{ is marked}}$ are asymptotically
independent. On the other hand $\{\frac{m_n}{N^{1/s}}\}_{n\text{ is marked}}$
are asymptotically Poisson by Lemma \ref{LmPVis}. In other words,
$$(\hTheta^{N,\delta}, \hat{\Gamma}^{N,\delta})=
\left(\left\{\frac{m_n}{N^{1/s}},\frac{\xi_n}{\rho_n}+\zeta_n\right\}_{n\text{ is }\delta-
\text{marked}}\right).$$
satisfy the condition of Theorem \ref{ThMain} except that (ii) is replaced by the condition

$(\widetilde{ii})$ $\Gamma^{N, \delta}_j$ converges weakly to a Poisson process on $[\delta, \infty)$
with measure $\mu^\delta.$ Moreover for each $[s,t]\in (0, \infty)$ 
$ \mu^\delta([s,u])\to \mu([s,u])$ as $\delta\to 0$ where
$\mu([s,u])=\int_s^u \frac{\bc}{x^{1+s}} dx$ (the last statement follows from
Lemma \ref{LmPTheta}).

This statement appears slightly weaker than we need but a general argument from real analysis allows
to upgrade $(\widetilde{ii})$ to (ii). 

Namely $(\widetilde{ii})$ shows that for each $\eps$ and $\brdelta$ there is a number
$N(\eps, \delta)$ such that for any collection of disjoint intervals
$[s_1, u_1],\dots, [s_l, u_l]$ in $[\brdelta, \infty)$ with 
$\mu([s_j,u_j])>\eps$ for each $k_1,\dots, k_l$ we
have
\begin{equation}
\label{PTMes}
\left|\bP\left(\Lambda_{N, \brdelta}(s_j, u_j)=k_j\ \hbox{ for all }j\in[1,l] \right)-
\prod_{j=1}^l \frac{\lambda_{j,\brdelta}^{k_j}}{k_j!} e^{-\lambda_{j, \brdelta}}\right|<\eps. 
\end{equation}
where $\Lambda_{N, \brdelta}(s_j, u_j)$ is the number of $\brdelta$-clusters and 
$\lambda_{j, \brdelta}=\mu^\brdelta([s_j, u_j]).$
Choose a sequence $\eps_k$ converging to $0$ (for example, $\eps_k=\frac{1}{k}$ will do)
and let $\delta_k$ be such that for each $\brdelta<\delta_k$ and each $[s, u]\in [\delta, \infty)$
we have
\begin{equation}
\label{EqLimInt}
|\mu^\brdelta([s,u])-\mu([s, u])|<\eps_m. 
\end{equation}
Then for $N>N(\eps_m, \delta_m)$ both 
\eqref{PTMes} and \eqref{EqLimInt} are valid. 
Let $k(N)$ be the largest number such that $N>N(\eps_k, \delta_k).$
Then \eqref{PTMes} and \eqref{EqLimInt} imply that
$$(\Theta^{N,\delta}, \Gamma^{N,\delta})=
\left(\left\{\frac{m_n}{N^{1/s}},\frac{\xi_n}{\rho_n}+\zeta_n\right\}_{n\text{ is }
\delta_{k(N)}-\text{ marked, } m_n>\delta N^{1/s} }\right)$$
satisfies (ii).

\section{Poisson Limit for expected occupation times.}
\label{ScPVis}

To understand the asymptotic properties of the distribution of $a_n$ defined in
\eqref{Defabm} we need the following
\begin{lemma}
\label{Lmmuinf}
(a) For each $m>0$ and $y\ge 1$
\begin{equation}\label{Claster}
\begin{aligned}
 \mu^m(y)&:=\lim_{N\to\infty} N
\bP\left(\frac{\rho_n}{N^{1/s}} \ge\delta y,
\rho_{n+1}< \delta N^{1/s}, \dots ,\rho_{n+m}< \delta N^{1/s}  \right)\\
&=\delta^{-s}c\bE[(D_{0}^{s}y^{-s}-\max_{1\le j\le m} D_{j}^{s})I_{\max_{1\le j\le m} D_{j} < D_{0}y^{-1}}],
\end{aligned}
\end{equation}
where $D_{j}:=p_{n+j}^{-1}\alpha_{n+j+1}...\alpha_{n+m}$ (and by convention
$D_{m}:=p_{n+m}^{-1}$).

\noindent (b) There exists
$\mu^\infty(y)=\lim_{m\to\infty} \mu^m(y).$

\noindent
(c) $\mu^\infty (1)>0. $
\end{lemma}
\begin{proof}
(a) We shall make use of \eqref{z2} and the relation $\rho_n=p_n^{-1}z_n$.
For a fixed $m$ and $0\le j\le m$ we can write
$$\rho_{n+j}=p_{n+j}^{-1}\alpha_{n+j+1}...\alpha_{n+m}z_{n+m} + \mathcal{O}(1).$$
The inequalities $\rho_n/{N^{1/s}} \ge\delta y$ and $\rho_{n+j}<\delta {N^{1/s}}$
in \eqref{Claster} are equivalent to
$$z_{n+m}\ge {N^{1/s}}\delta(D_{0}^{-1} y+ \mathcal{O}(N^{-1/s}))
\hbox{ and } z_{n+m}< {N^{1/s}}\delta (D_{j}^{-1}+ \mathcal{O}(N^{-1/s}))$$
respectively. Thus
$$\begin{aligned}
&\bP\left(\frac{\rho_n}{N^{1/s}} \ge \delta y, \rho_{n+1}< \delta N^{1/s}, \dots ,
\rho_{n+m}< \delta N^{1/s}\right)=\\
&\bP\left(\delta(D_{0}^{-1} y +\mathcal{O}(N^{-1/s})){N^{1/s}}
\le z_{n+m}<{N^{1/s}}\delta\min_{1\le j\le m} (D_{j}^{-1}
 +\mathcal{O}(N^{-1/s})) \right).
\end{aligned}$$
Since $z_{n+m}$ and $\{p_{n+j}\}_{j\le m}$ are independent, we can compute the
following limit by conditioning on $\{p_{n+j}\}_{j\le m}$ and using Lemma \ref{LmRen}:
$$\begin{aligned}
&\lim_{N\to\infty}N
\bP\left({\rho_n}{N^{-1/s}} \ge y, \rho_{n+1}< \delta N^{1/s}, \dots ,
\rho_{n+m}< \delta N^{1/s}| \{p_{n+j}\}_{j\le m}\right)=\\
&\delta^{-s}c(D_{0}^{s}y^{-s}-\max_{1\le j\le m} D_{j}^{s})I_{\max_{1\le j\le m} D_{j} < D_{0}y^{-1}}.
\end{aligned}$$
To compute the limit \eqref{Claster}, it remains to take the expectation with respect to $\{p_{n+j}\}_{j\le m}$:
$$\begin{aligned}
&\lim_{N\to\infty}N
\bP\left({\rho_n}{N^{-1/s}} \ge y, \rho_{n+1}< \delta N^{1/s}, \dots ,
\rho_{n+m}< \delta N^{1/s}\right)=\\
&\delta^{-s}c\bE[(D_{0}^{s}y^{-s}-\max_{1\le j\le m} D_{j}^{s})I_{\max_{1\le j\le m} D_{j} < D_{0}y^{-1}}].
\end{aligned}$$
This completes the proof of part (a).\newline\smallskip
(b) The probability $\quad
\bP\left(\frac{\rho_n}{N^{1/s}} \in [c,d], \rho_{n+1}\leq \delta N^{1/s}, \dots ,
\rho_{n+m}\leq \delta N^{1/s} \right)$
is a monotonically decaying function of $m$. Hence the proof.

\smallskip\noindent
(c) If $\mu^\infty(1)=0$ then $N \bP(n \text{ is marked})\to 0$ as $N\to\infty.$
Then
$$
\begin{aligned}
\bP(\rho_n\geq \delta N^{1/s})&\le \bP(\rho_n\geq \delta N^{1/s} \hbox{ and $n$ is not in a cluster})\\
&+\bP(\rho_n\geq \delta N^{1/s} \hbox{ and $n$ is in a cluster})\\
& \le \Const \frac{\beta^M}{N}+ \Const\bP(n \hbox{ is marked}),
\end{aligned}
$$
where the estimates for
the first and second term are provided by Lemmas \ref{non-cluster}
and \ref{in-cluster} respectively.
But then $N\bP(\rho_n\geq \delta N^{1/s})\to0$ as $N\to\infty$
contradicting Lemma \ref{LmRen}. This proves (c).
\end{proof}
Lemma \ref{Lmmuinf} gives the limiting distribution of $\ta$ in Lemma \ref{LmPVis}.
Namely, $\bP(\ta>y)=1$ if $y\le1$ and for $y>1$ we have
$$
\bP(\ta>y)=\lim_{N\to\infty}\bP(\rho_n>N^{1/s}\delta y\, |\, n\hbox{ is marked})=
\mu^{\infty}(y)/\mu^{\infty}(1).
$$
Next we address the distribution of $\tb_j.$
\begin{lemma}
The distribution of $\frac{\sum_{j=0}^M \rho_{n-j}}{\rho_n}$ conditioned on
$\rho_n\geq \delta N^{1/s}$ converges as $N\to\infty$ to the distribution
of
$$
1+p_{-1}^{-1}q_0+ p_{-2}^{-1}q_0\alpha_{-1}+\dots+p_{-k}^{-1}q_0\alpha_{-1}\dots \alpha_{-k+1}+...
$$
\end{lemma}
\begin{proof} According to \eqref{SManySteps}
$$ \rho_{n-j}=p_{n-j}^{-1}q_n\alpha_{n-1}\dots \alpha_{n-j+1}\rho_n+\cO(K^M).$$
Since $K^M\ll N^{1/s}$ we see that
$$ \frac{\sum_{j=0}^M \rho_{n-j}}{\rho_n}=1+p_{n-1}^{-1}q_n+ p_{n-2}^{-1}q_n\alpha_{n-1}+\dots+p_{n-M}^{-1}q_n\alpha_{n-1}\dots \alpha_{n-M+1}+o(1).$$
As $N\to\infty$, also $M=M_N\to\infty$ and so the limiting distribution of the above expression is the same as the distribution of
\begin{equation}
\label{StatDist}
1+p_{-1}^{-1}q_0+ p_{-2}^{-1}q_0\alpha_{-1}+\dots+p_{-k}^{-1}q_0\alpha_{-1}\dots \alpha_{-k+1}+...
\end{equation}
\end{proof}
Next take $\eps_5<\eps_4<\eps_2$ where $\eps_2$ is from Lemma \ref{Lm2Clusters}(b).
Divide all possible values of $a_n$ into intervals $I_1, I_2, \dots, I_{d_1}.$ Divide $[0, N]$ into
a union of long intervals $L_j$ of length  $N^{\eps_4}$ and short intervals of  length $N^{\eps_5}.$
(Intervals are numbered in decreasing order). Then by Lemma \ref{LmRen} the total number of clusters
originated in short intervals tends to $0$ in probability. Observe that by Lemmas \ref{Lm2Clusters} and \ref{Lmmuinf}
$$
\begin{aligned}&\bP\left(n \text{ is marked, } a_n\in I_m \text{ and }
\rho_{n-k}\leq \delta N^{1/s}, k=M\dots N^{\eps_4}\right)\\
&\sim \frac{\mu_\infty(I_m)}{N} \left(1-\cO\left(\beta^M\right)\right).
\end{aligned}
$$
Recall that $b_n$ is independent of $a_n.$ Hence if we divide
$[1,\infty)\times [1,\infty)$ into rectangles $J_1, J_2\dots J_{d_1}$  then
$$
\begin{aligned}&\bP\left(n \text{ is marked, } (a_n, b_n) \in J_m \text{ and }
\rho_{n-k}\leq \delta N^{1/s}, k=M\dots N^{\eps_4}\right)\\
&\sim \frac{\tmu_\infty(J_m)}{N} \left(1-\cO\left(\beta^M\right)\right)
\end{aligned}
$$
where $\tmu_\infty$ is a product of $\mu_\infty$ and the distribution function of $\tb_n.$

Let $V_j$ be the vector whose $m$-th component is
$$\Card(n\in L_j: n-marked, (a_n, b_n)\in J_m). $$
Then
$$ \bP(V_j=e_m)\sim \tmu(J_m) N^{\eps_4-1}, \quad
\bP(|V_j|\geq 1)=o(N^{\eps_4-1})$$
so
\begin{equation}
\label{1Int}
\bE(\exp(i\langle v, V_j\rangle)=
1+N^{\eps_4-1} \sum_m \tmu(J_m) (e^{iv_m}-1)+o(N^{\eps_4-1}).
\end{equation}
Next, let $W_j=\sum_{k=1}^j V_k.$ We claim that
\begin{equation}
\label{kInt}
\ln \bE(\exp(i\langle v, W_j\rangle)=j N^{\eps_4-1} \sum_m \tmu(J_m) (e^{iv_m}-1)+o(j N^{\eps_4-1}).
\end{equation}
This holds because $V_j$ is almost independent of $V_1, V_2\dots V_{j-1}.$
Namely, by Lemma \ref{LmCouple}
the value of $\rho_n$ at the left endpoint of $L_j$ could influence $V_j$
only if $\rho_{n-k}$ is $\beta_1^{N^{\eps_5}}$-close to the boundary of $I_m.$ However if $N$
is large then the probability that there is $n-k\in L_j$ such $\rho_{n-k}$ is close to the boundary
of $I_m$ is $o(N^{\eps_4-1})$ and hence arguing as in the proof of \eqref{1Int} we obtain \eqref{kInt}.
Taking $j\sim N^{1-\eps_4}$ we obtain Lemma \ref{LmPVis}.

\section{Case $s=2$: proof of Theorem \ref{ThQLT}}
\label{ScS2} To prove Theorem \ref{ThQLT} we follow the approach
used in \cite{CD}.

We split
$$ \sum_{n=1}^N (\xi_n-\rho_n)=S_L+S_M+S_H $$
where
$S_H$ corresponds to the high values of $\rho_n,$ namely, $\rho_n>\sqrt{N} \ln^{100} N,$
$S_M$ corresponds to the moderate values of $\rho_n,$ namely,
$\frac{\sqrt{N}}{\ln^{100} N}\leq  \rho_n\leq \sqrt{N} \ln^{100} N$ and
$S_L$ corresponds to the low values of $\rho_n,$ namely, $\rho_n<\frac{\sqrt{N}}{\ln^{100} N}.$
We begin by showing that high and moderate values of $\rho_n$ can be ignored.
First, by Lemma \ref{LmRen}
$$\Prob(S_H\neq 0)\leq N \Prob(\rho_n>\sqrt{N}\ln^{100} N)\leq \frac{C}{\ln^{200} N}. $$
Second, arguing as in the proof of Lemma \ref{LmLow}(b) we see that
$$\bE(S_M^2)\leq \Const \sum
\bE\left(\left(\rho_n\right)^2I_{\sqrt{N}/\ln^{100} N<\rho_n<\sqrt{N}\ln^{100} N} \right)\leq
\Const N \ln\ln N$$
and hence $S_M/\sqrt{N\ln N}$ converges to $0$ in probability.

Therefore the main contribution comes from $S_L.$ To handle it use Bernstein's method.
Divide the interval $[0, N]$ into blocks of length
$L_N=\ln^{10} N$ and $l_N=\ln^2 N$ following each other. More precisely
the $j$-th big block is
$$ I_j=[j(L_N+l_N), (j+1)L_N+j l_N-1]$$
and $j$-th small block is
$$ J_j=[(j+1)L_N+j l_N, (j+1)(L_N+l_N)-1].$$
Accordingly, we split $S_L=S_L^{big}+S_L^{small},$ where
$S_L^{big}$ ($S_L^{small}$) is the contribution to $S_L$ coming from big (small) blocks.
Arguing as in the proof of Lemma \ref{LmLow}(b)
we see that
$$ \bE(\Var_\omega(S_L^{small}))\leq
 C\sum_n \left[\bE((\rho_n)^2)+\bE(\rho_n l_N)\right] \text{ (summation is over the small blocks)}$$
$$ \leq C\left(N \ln N \frac{l_N}{L_N}+ N \frac{l_N^2}{L_N}\right) $$
and hence the main contribution comes from the big blocks.

Next we modify $\xi_n$ as follows. If $n\in I_j$ let $\txi_n$ be the  number of visits to
the site $n$ before our walk reaches $I_{j+1}.$ Let $\trho_n=\EXP_\omega(\txi_n).$
Observe that $\txi_n$ corresponds to imposing absorbing boundary conditions at
the beginning of $I_{j+1}$ so
$\trho_n=p_n^{-1} q_{n+1} \trho_n+p_n^{-1}$
with absorbing boundary
condition at $\brn:=(L_N+l_N)(j+1).$ Hence
$$ \rho_n-\trho_n=\frac{q_\brn}{q_n} \alpha_n\dots \alpha_{\brn-1} \rho_\brn $$
and so
$$ \left|\sum_n \left[\trho_n-\rho_n\right]\right|<1 $$
except for the set of probability tending to $0$ as $N\to\infty.$ Also
$$ \bP(\txi_n=\xi_n \text{ for } n=0\dots N)\to 1 \text{ as }N\to\infty.$$
Let
$$\tS=\sum_n (\txi_n-\trho_n) $$
where the sum is over big blocks.
By the foregoing discussion it is enough to show that
\begin{equation}
\label{QNorm}
\text{with $\bP$ probability close to 1
the quenched distribution of $\tS$ is close to normal.}
\end{equation}
We claim that the following limit
exists (in probability)
\begin{equation}
\label{LowDif}
\lim_{N\to\infty} \frac{\Var_\omega(\tS)}{N\ln N}=D_2.
\end{equation}
Before proving \eqref{LowDif} let us show ho to complete the proof of \eqref{QNorm}. Let
$$\tS_j=\sum_{n\in I_j} (\txi_n-\trho_n) $$
be the contribution of the $j$-th block.
Since summation is taken over $n$ with $\rho_n<\sqrt{N}/\ln^{100} N$ and $\txi_n$ has geometric distribution
we have for $k\in \naturals$
\begin{equation}
\label{PreLin}
\Prob_\omega\left(\tS_j>\frac{\sqrt{N} L_N k}{\ln^{100} N}\right)\leq C e^{-k} L_N.
\end{equation}
Indeed $\tS_j>\frac{\sqrt{N} L_N k}{\ln^{100} N}$ implies that $\txi_n>\frac{\sqrt{N} k}{\ln^{100} N}$
for some $n$ in the block. \eqref{LowDif} and \eqref{PreLin} show that $\sum_j \tS_j$ satisfies the Lindenberg condition.
It remains to establish
\eqref{LowDif}. To this end we prove two facts.

$$(A)\quad \forall\eps>0 \exists M: \bP\left(
\frac{\sum_{n_1<n_2-M} \Cov_\omega(\txi_{n_1}, \txi_{n_1})}{N\ln N}>\eps\right)<\eps \quad \text{ and }$$
$$(B)\quad \forall k\quad
\frac{\sum_n \Cov_\omega(\txi_{n}, \txi_{n-k})}{N\ln N}\Rightarrow \frac{\bE(\alpha)^k c^*}{2}
\text{ in probability} $$
where $c^*$ is the constant from Lemma \ref{LmRen}.

The remaining part of Section \ref{ScS2} is devoted to the proofs of statements (A) and (B). We will drop
tildes in $\txi$ and $\trho$ in order to simplify notation.

To obtain (A) we  show that
\begin{equation}
\label{Covs>1}
\bE(\left|\Cov_\omega\left(\xi_{n-k}, \xi_{n}\right)\right|\left|\cF_{n}\right.)\leq C \theta^k (\rho_{n})^2
\end{equation}
for some $\theta<1.$ Pick a small $\epsilon>0$ and consider two cases

(I) $\rho_{n}>(1+\epsilon)^k.$ Then we use that
$$ |\Cov_\omega(\xi_{n-k}, \xi_{n})|\leq \sqrt{\Var_\omega(\xi_{n-k})\Var_\omega(\xi_{n})}
\leq C \rho_{n_1} \rho_{n_2}$$
and that
$$\bE(\rho_{n-k}|\cF_n)\leq \bE(\alpha)^k \rho_n+C.$$
(II) $\rho_{n}\leq (1+\epsilon)^k.$ Then by \eqref{Cor1MC}
$$ |\Cov_\omega(\xi_{n-k}, \xi_{n})|\leq C \Var_\omega(\xi_{n}) \rho_{n-k} q^*$$
where $q^*$ is the  probability to visit $n-k$ before $n$ starting from $n-1.$ Hence
$$ \bE(|\Cov_\omega(\xi_{n-k}, \xi_{n})||\cF_n)\leq C \rho_n
\sqrt{\bE((\rho_{n-k})^2|\cF_n) \bE(q^*|\cF_n)} $$
We have
$$\bE(\rho_{n-k})^2|\cF_n)\leq \rho_n+Ck$$
since $s=2$ whereas $\bE(q^*|\cF_n)\leq C \theta^k $ by Lemma \ref{LmBack}.

Summing \eqref{Covs>1} over $k$ we obtain (A).

To prove (B) observe that by Lemma \ref{LmCorShort} for fixed $k$ we have
$$\Cov_\omega(\xi_{n-k}, \xi_n)=\rho_{n-k} \rho_{n}+\cO\left(\rho_n \right)$$
where the implicit constant depends on $k.$ Since
$\bE(\rho_{n-k}|\cF_n)=\rho_n \bE(\alpha)^k+C$ we get
$$\bE(\Cov_\omega(\xi_{n-k}, \xi_n))=
\bE((\rho_{n})^2) \bE(\alpha)^k+\cO\left(\bE\left(\rho_n\right)\right).$$
Let $Z_n=\Cov(\txi_n, \txi_{n-k}).$
Next
$$\Var\left(\sum_n Z_n\right)=\sum \Var(Z_n)+2\sum_{n_1<n_2} \Cov(Z_{n_1}, Z_{n_2}).$$
Observe that $Z_{n_1}$ and $Z_{n_2}$ are independent if $n_1, n_2$ belong to different
blocks and so we can limit summation over $n_1, n_2$ in the same block.
Since
$$Z_n=\rho_n^2 \alpha_{n-k}\dots \alpha_{n-1}+\cO(\rho_n)$$
Lemma \ref{LmRen} gives
$$ \Var(Z_n)\leq \Const \frac{N}{\ln^{200} N}.$$
By Cauchy-Schwartz inequality
$$ \Cov(Z_{n_1}, Z_{n_2}) \leq \Const \frac{N}{\ln^{200} N}.$$
Therefore
$$ \Var(\sum_n Z_n)\leq \Const \frac{N L_N^2}{\ln^{100} N}.$$
This completes the proof of (B).

\begin{remark}
The same argument allows one to handle the case $s>2.$ Actually this case is simper
since there is no need to introduce cutoffs. We do not provide the details here since the case $s>2$
had been studied in detail in \cite{G2, P1}\footnote{\cite{G2} considers a class
of environments which is much larger than the one treated in our paper.}, where stronger almost sure quenched
limit theorems were obtained (as has already been mentioned in the Introduction). However, such almost sure statement can not be extended to $s=2$ for two (related) reasons. First
the quenched variance of $\rho_n$ is not integrable and so \eqref{LowDif} can not be upgraded to almost
sure convergence \cite{Aa}. Secondly even though the contribution of the site with largest $\rho_n$
is much smaller than the contribution of the remaining sites with probability close to 1, still
$\bP(\max_n \rho_n>\sqrt{N}{\ln^{100} N})$ decays quite slowly (as $\ln^{-200} N$) and so from time to time
we will see the situation where the site with largest $\rho_n$ can not be ignored: the distribution of the sequence $T_N$ would alternate between that of the normal and exponential variables. So our method does not recover the results of \cite{G2, P1} but it allows us to reprove Theorem \ref{ThAnn} for all values of $s$ (as in \cite{KKS}).
\end{remark}


\section{Maximum occupation time.}
\label{ScMOT}
Here we prove Theorem \ref{ThRec}. Consider the following process
$\hLambda^\delta_N=\{(\frac{n_j}{N}, \frac{\hm_j}{N^{1/s}})\}$ where
$n_j$ are marked points and $\hm_j$ is the maximum of $\rho_n$ inside the $j$-th cluster.
Similarly to Lemmas \ref{LmPVis} and \ref{LmPTheta} we prove the following statement.

\begin{lemma}
(a) As $N\to\infty$ $\hLambda^\delta_N$ converges to a Poisson process
$\hLambda^\delta=\{t_j, \bm_j)\}$
on $[0,1]\times [\delta, \infty).$

(b) As $\delta\to 0$ $\hLambda^\delta$ converges to the Poisson process $\hLambda.$

(c) There exists a constant $\hc$ such that $\hLambda$ has the intensity
$\frac{\hc}{\bm^{1+s}}.$
\end{lemma}

Next we show that the low values of $\rho$ are unlikely to contribute to the maximal occupation times.
Fix $\theta>0.$ Denote
$$\Omega_{N,k}=\{\exists n\leq N: N^{1/s}2^{-(k+1)}<\rho_n\leq N^{1/s}2^{-k}
\text{ and }
\xi_n>\theta N^{1/s}\}$$
and set
$$
\Phi_{N,k,n}=\{ N^{1/s}2^{-(k+1)}<\rho_n\leq N^{1/s}2^{-k}\}.
$$
Then by Lemma \ref{LmRen}
$$ \bP(\Omega_{N,k})\le N
\bP\left(\Phi_{N,k,n} \right)
\bP\left(\xi_n>\theta N^{1/s}| \Phi_{N,k,n}\right)\le \Const 2^{ks}
\bP\left(\xi_n>\theta N^{1/s}| \Phi_{N,k,n}\right)$$
Since $\xi_n$ has a geometric distribution with parameter $\rho_n^{-1}$ we have that
$$
\bP\left(\xi_n>\theta N^{1/s}| \Phi_{N,k,n}\right)\le (1-\rho_n^{-1})^{\theta N^{1/s}}\le
\Const e^{-c2^k}.
$$
The first term here is $\cO(2^{-ks})$ in view of Lemma \ref{LmRen} and Markov inequality and the
second term is less than
$$ 4^{sk} \bP\left(\xi_n>\theta N^{1/s}|\rho_n \leq \frac{N^{1/s}}{2^k}\right)
\leq \Const 4^{sk} \beta^{2^k \theta}, \quad \beta<1 $$
since $\xi_n$ has geometric distribution with mean $\rho_n.$
Summing these bounds over $k\geq \log_2(1/\delta)$
we see that the points from outside of the clusters
can be ignored. The rest of the proof of Theorem \ref{ThRec}
is similar to the proof of Theorem
\ref{ThMain}. Namely Lemma \ref{LmCorShort} implies that the maximum occupation
time inside the
$j$-th cluster occurs at the site $\hn_j$ such that $\rho_{\hn_j}=\hm_j.$ This shows that
if $\delta$ is sufficiently small then
with probability close to 1
$\xi_N^{*}=\max_j \hm_j \frac{\xi_{\hn_j}}{\hm_j}$ where the maximum is taken
over the $\delta$-clusters.
For large $N$ the $\frac{\xi_{\hn_j}}{\hm_j}$
is asymptotically exponential with mean 1. Therefore letting
$N\to\infty$ and $\delta_N\to 0$ we obtain that the distribution of
$\frac{\xi^*_N}{N^{1/s}}$
is asymptotically the same as that of
$$ \max_j \htheta_j \Gamma_j $$
where $\hLambda=\{(t_j, \htheta_j)\}$ and $\Gamma_j$ are i.i.d random variables independent of
$\hLambda$ and having mean 1 exponential distribution. It remains to notice that by Lemma~\ref{LmPT}
$\{\htheta_j \Gamma_j\}$ also form a Poisson process.

\appendix
\section{Annealed distribution.}
\label{AppAnn}
Here we show how our results allow to recover the known facts about the annealed distribution.
\subsection{Proof of Theorem \ref{ThAnn}.}
If $0<s<1$ then our result follows from Theorem \ref{ThQLT}(a),
Lemma \ref{LmPT}(c) and Lemma \ref{LmPSt}(a).

If $1<s<2$ let
$$ Y_\delta'=\sum_{\Theta_i>\delta} \Theta_i \Gamma_i, \quad
Y_\delta''=\sum_{\Theta_i>\delta} \Theta_i ,
\quad Y_\delta=Y_\delta'-Y_\delta''. $$
Observe that $\bE(Y_\delta')=\bE(Y_\delta'').$
By Theorem \ref{ThQLT}(b) $T_N-\bE(T_N)$ is asymptotically distributed as
$$ Y_\delta+(Y_\delta''-\bE Y_\delta'')=(Y_\delta'-\bE Y_\delta').$$
Therefore the result follows by
Lemma \ref{LmPT}(c) and Lemma \ref{LmPSt}(b).

The proofs in cases $s=1,2$ are similar.

\begin{remark}
In this paper we restrict our attention to the case $s\leq 2.$ We refer the reader to
\cite{G1} for the analysis of case $s>2.$
\end{remark}

\subsection{Proof of Corollary \ref{CrNoLimit}.}
Fix a metric $d$ on the space of distributions of the line.
For example, one can take
$$ d(F_1, F_2)=\inf\{\eps: F_2(x-\eps)-\eps<F_1(x)<F_2(x+\eps)+\eps\}. $$
To prove Corollary \ref{CrNoLimit} it suffices to show that given $a_1, a_2\dots a_l$
and $\eps>0$ the event
$$ d(F_N^\omega, F_{\sum_{j=1}^l a_j \Gamma_j})<\eps $$
occurs infinitely often where $\Gamma_1\dots \Gamma_l$ are i.i.d.
random variables having exponential distribution with parameter 1.

Given an interval $I=[n_1, n_2]$ let $T_I$ be the total time the walker
spends inside $I$ before $\tilde{T}_{n_2},$ the hitting time of $n_2$.
Note that the quenched distribution functions
$F_{T_{I_1}}^\omega \dots F_{T_{I_k}}^\omega$ of $T_{I_j}$ are independent if
the intervals $I_1, I_2\dots I_k$ have disjoint interiors.
On the other hand we can choose a sequence $N_m$ growing so fast
that
\begin{equation}
\label{ForgetPast}
\bP\left(d(F_{N_m}^\omega, F_{T_{N_{m-1}, N_m}/N_m^{1/s}})<\frac{\eps}{4}\right)<\frac{1}{m^{10}}.
\end{equation}
Accordingly in view of Borel-Cantelli Lemma it is enough to show that the event
\begin{equation}
\label{ComeNear}
 d(F_{\sum_{j=1}^l a_j \Gamma_j}, F_{T_{N_{m-1}, N_m}/N_m^{1/s}})<\frac{\eps}{2} \quad
\text{occurs infinitely often.}
\end{equation}
By Corollary \ref{CrQLT} there exists $c=c(\eps)>0$ such that
$$ \bP\left(d(F_{N_m}^\omega, F_{\sum_{j=1}^l a_j \Gamma_j})<\frac{\eps}{4}\right)>c. $$
Now \eqref{ForgetPast} implies that for large $m$ we have
$$\bP(\left( d(F_{\sum_{j=1}^l a_j \Gamma_j}, F_{T_{N_{m-1}, N_m}/N_m^{1/s}})<\frac{\eps}{2} \right)>\frac{c}{2}$$
and hence \eqref{ComeNear} follows from Borel-Cantelli Lemma.

\subsection{Proof of Corollary \ref{CrRec}.}
Let $N_k$ be a sequence growing so fast that
\begin{equation}
\label{EqFastGrowth}
\sum_k \bP(X_n \text{ visits } [0, N_k] \text{ after } \tilde{T}_{N_{k+1}})<\infty.
\end{equation}
Let $\tX_n$ be a modified walk obtained from $X_n$ by erasing visits to $[0, N_k]$ after $\tilde{T}_{N_{k+1}}.$
By \eqref{EqFastGrowth} $\tX_n$  is a finite modification of $X_n.$ Let $\txi_n$ and $\txi_N^*$
be defined similarly to $\xi_n$ and $\xi_N^*$ but with $X_n$ replaced by $\tX_n.$ Observe that
if $N_k$ grows sufficiently fast then by Theorems \ref{ThAnn} and \ref{ThRec} there exist constants
$c_1$ and $c_2$ such that
$$ \bP\left(\frac{\txi^*_{N_k}}{\tilde{T}_{N_k}}>c_1\right)>c_2. $$
Therefore Corollary \ref{CrRec} follows by Borel-Cantelli Lemma.

\end{document}

%% file: abbrev.tex
\def\eps{{\varepsilon}}

\def\Card{{\rm Card}}
\def\Const{{\rm Const}}
\def\Corr{{\rm Corr}}
\def\Cov{{\rm Cov}}

\def\Prob{{\mathbb{P}}}
\def\Var{{\rm Var}}

\def\EXP{{\mathbb{E}}}

\def\naturals{\mathbb{N}}

\def\reals{\mathbb{R}}

\def\integers{\mathbb{Z}}

\def\bE{\mathbf{E}}

\def\bP{\mathbf{P}}

\def\bc{\mathbf{c}}

\def\bm{\mathbf{m}}

\def\bsigma{\boldsymbol{\sigma}}

\def\brC{{\bar C}}

\def\brrC{{\bar\brC}}

\def\brc{{\bar c}}

\def\brbc{\mathbf{\brc}}

\def\brn{{\bar n}}

\def\brq{{\bar q}}

\def\brp{{\bar p}}

\def\brrp{{\bar{\brp}}}

\def\brq{{\bar q}}

\def\brrq{{\bar{\brq}}}

\def\brbeta{{\bar \beta}}

\def\brdelta{{\bar \delta}}

\def\breta{{\bar \eta}}

\def\brreta{{\bar{\breta}}}

\def\brxi{{\bar \xi}}

\def\brTheta{{\bar \Theta}}

\def\brOmega{{\bar\Omega}}

\def\cF{\mathcal{F}}

\def\cN{\mathcal{N}}

\def\cO{\mathcal{O}}

\def\hc{{\hat c}}

\def\hm{{\hat m}}

\def\hn{{\hat n}}

\def\hLambda{{\hat\Lambda}}

\def\hrho{{\hat\rho}}

\def\htheta{{\hat\theta}}

\def\hTheta{{\hat\Theta}}

\def\tC{{\tilde C}}

\def\tK{{\tilde K}}

\def\tS{{\tilde S}}

\def\tX{{\tilde X}}

\def\tY{{\tilde Y}}

\def\ta{{\tilde a}}

\def\tb{{\tilde b}}

\def\tc{{\tilde c}}

\def\tf{{\tilde f}}

\def\tdelta{{\tilde\delta}}

\def\tmu{{\tilde\mu}}

\def\trho{{\tilde\rho}}

\def\tTheta{{\tilde\Theta}}

\def\txi{{\tilde\xi}}

\def\tOmega{{\tilde{\Omega}}}

\makeatother

\def\beq{\begin{equation}}
\def\eeq{\end{equation}}